\documentclass[a4paper]{amsart}
\usepackage[utf8x]{inputenc}
\usepackage{graphicx,color,easybmat}
\usepackage{amssymb,amsmath}

\newtheorem{MainTheorem}{Theorem}

\newtheorem*{conjecture*}{Conjecture}
\newtheorem{MainCorollary}[MainTheorem]{Corollary}

\newtheorem{theorem}{Theorem}[section]
\newtheorem{proposition}[theorem]{Proposition}
\newtheorem{lemma}[theorem]{Lemma}
\newtheorem{corollary}[theorem]{Corollary}
\theoremstyle{definition}
\newtheorem{remark}[theorem]{Remark}
\newtheorem{defi}[theorem]{Definition}
\newtheorem{example}[theorem]{Example}
\newenvironment{case}[2][Case]{%
  \trivlist \item[\hskip\labelsep{\itshape #1 #2.}]\begin{em}}{
  \end{em}\endtrivlist}

\newcommand{\evalat}[1]{\bigr\rvert_{#1}}
\newcommand{\set}[2]{\ensuremath{\{#1 \,\colon #2\}}}
\usepackage{enumerate}
\makeatletter
\renewcommand\theenumi{\@alph\c@enumi}
\renewcommand\theenumii{\@alph\c@enumii}
\renewcommand\theenumiii{\@alph\c@enumiii}
\renewcommand\theenumiv{\@alph\c@enumiv}

\def\@map#1#2[#3]{\mbox{$#1 \colon #2 \longrightarrow #3$}}
\def\map#1#2{\@ifnextchar [{\@map{#1}{#2}}{\@map{#1}{#2}[#2]}}
\def\@chull#1[#2]{\ensuremath{\langle #1 \rangle_{_{#2}}}}
\def\chull#1{\@ifnextchar [{\@chull{#1}}{\ensuremath{\langle #1 \rangle}}}
\makeatother

\DeclareMathOperator{\Int}{Int}
\DeclareMathOperator{\Pos}{Pos}
\DeclareMathOperator{\Irr}{Irr}
\DeclareMathOperator{\Cl}{Cl}

\DeclareMathOperator{\En}{En}

\begin{document}
\title[Minimum positive entropy for tree cycles in any period]{Characterization of the tree cycles with minimum positive entropy for any period}
\author{David Juher \and Francesc Ma\~{n}osas \and David Rojas}

\address{Departament d'Inform\`atica, Matem\`atica Aplicada i Estad\'{\i}stica,
Universitat de Gi\-ro\-na, c/ Maria Aur\`elia Capmany 61, 17003 Girona, Spain.
ORCID 0000-0001-5440-1705}
\email{david.juher@udg.edu \textrm{(Corresponding author)}}

\address{Departament de Matem\`atiques, Edifici C, Universitat
Aut\`onoma de Barcelona, 08913 Cerdanyola del Vall\`es, Barcelona, Spain. ORCID 0000-0003-2535-0501}
\email{manyosas@mat.uab.cat}

\address{Departament d'Inform\`atica, Matem\`atica Aplicada i Estad\'{\i}stica,
Universitat de Gi\-ro\-na, c/ Maria Aur\`elia Capmany 61, 17003 Girona, Spain. ORCID 0000-0001-7247-4705}
\email{david.rojas@udg.edu}

\thanks{This work has been funded by grants PID2020-118281GB-C31 of Ministerio de Ciencia e Innovaci\'on and 2021 SGR 00113 of Ge\-ne\-ra\-litat de Catalunya. D.R. is a Serra H\'unter fellow.}


\begin{abstract}
Consider, for any integer $n\ge3$, the set $\Pos_n$ of all $n$-periodic tree patterns
with positive topological entropy and the set $\Irr_n\subset\Pos_n$ of all $n$-periodic
irreducible tree patterns. The aim of this paper is to determine the elements of minimum
entropy in the families $\Pos_n$, $\Irr_n$ and $\Pos_n\setminus\Irr_n$. Let $\lambda_n$ be the unique
real root of the polynomial $x^n-2x-1$ in $(1,+\infty)$. We explicitly construct an
irreducible $n$-periodic tree pattern $\mathcal{Q}_n$ whose entropy is $\log(\lambda_n)$.
We prove that this entropy is minimum in $\Pos_n$. Since the pattern $\mathcal{Q}_n$ is
irreducible, $\mathcal{Q}_n$ also minimizes the entropy in the family $\Irr_n$. We also prove
that the minimum positive entropy in the set $\Pos_n\setminus\Irr_n$ (which is nonempty only
for composite integers $n\ge6$) is $\log(\lambda_{n/p})/p$, where $p$ is the least prime
factor of $n$.
\end{abstract}

\maketitle

\section{Introduction}\label{S1}
The field of Combinatorial Dynamics has its roots in the striking Sharkovskii's Theorem \cite{shar}, in the sense that the theory
grew up as a succession of progressive refinements and generalizations of the ideas contained in the original proof of that result.
The core of the theory is the notion of \emph{combinatorial type} or \emph{pattern}.

Consider a class $\mathcal{X}$ of topological spaces (closed intervals of the real line, trees, graphs and compact surfaces
are classic examples) and the family $\mathcal{F}_\mathcal{X}$ of all maps $\{\map{f}{X}:X\in\mathcal{X}\}$ satisfying a given
property (continuous maps, homeomorphisms, etc). Any of such maps gives rise, by iteration, to a discrete dynamical system. Assume
now that we have a map $\map{f}{X}$ in $\mathcal{F}_\mathcal{X}$ which is known to have a periodic orbit $P$. The \emph{pattern of $P$}
is the equivalence class $\mathcal{P}$ of all maps $\map{g}{Y}$ in $\mathcal{F}_\mathcal{X}$ having an invariant set $Q\subset Y$ that,
at a combinatorial level, behaves like $P$. In this case, we say that every map $g$ in the class \emph{exhibits} the pattern $\mathcal{P}$.
Of course we have to precise in which sense a periodic orbit \emph{behaves as $P$}. So, we have to decide which feature of $P$ has to be
preserved inside the equivalence class $\mathcal{P}$. The period of $P$, just a natural number, is a first possibility (Sharkovskii's Theorem),
but a richer option arises from imposing that
\begin{enumerate}
\item the relative positions of the points of $Q$ inside $Y$ are the same as the relative positions of $P$ inside $X$
\item the way these positions are permuted under the action of $g$ coincides with the way $f$ acts on the points of $P$.
\end{enumerate}

An example is given by the family $\mathcal{F}_\mathcal{M}$ of surface homeomorphisms. The pattern (or \emph{braid type}) of a cycle
$P$ of a map $\map{f}{M}$ from $\mathcal{F}_\mathcal{M}$, where $M$ is a surface, is defined by the isotopy class, up to conjugacy,
of $f\evalat{M\setminus P}$ \cite{bow,mat}.

When $\mathcal{F}_\mathcal{X}$ is the family of continuous maps of closed intervals, the points of an orbit $P$ of a map in $\mathcal{F}_\mathcal{X}$
are totally ordered and the pattern of $P$ can be simply identified with a cyclic permutation in a natural way. The notion of pattern for
interval maps was formalized and developed in the early 1990s \cite{bald,mn}.

In the last decades, a growing interest has arisen in extending the notion of \emph{pattern} from the interval case to more general
one-dimensional spaces such as graphs \cite{patgraf,AMM} or trees \cite{aglmm,bald2,bern}. Precisely, in this paper we deal with patterns of
periodic orbits of continuous maps defined on trees (simply connected graphs).

Let us precise the conditions (a,b) above in our context. If $\map{f}{T}$ is a continuous map of a tree and $P\subset T$ is a periodic orbit of $f$,
the triplet $(T,P,f)$ will be called a \emph{model}. Two points $x,y$ of $P$ will be said to be \emph{consecutive} if the unique closed interval of
$T$ having $x,y$ as endpoints contains no other points of $P$. Any maximal subset of $P$ consisting only of pairwise consecutive points
will be called a \emph{discrete component}. We will say that two models $(T,P,f)$ and $(T',P',f')$
are equivalent if there is a bijection $\phi$ from $P$ to $P'$ which sends discrete components to
discrete components and conjugates the action of $f$ on $P$ and the action of $f'$ on $P'$, i.e. $f' \circ \phi\evalat{P} = \phi \circ f\evalat{P}$.
In Figure~\ref{expat} we show two equivalent 6-periodic models
with two discrete components. Note that two points $x_i,x_j$ of $P$ are consecutive
in $T$ when the corresponding points $x'_i,x'_j$ of $P'$ are consecutive in $T'$.

A \emph{pattern} is an equivalence class of models by the above equivalence relation. A map $\map{f}{T}$ is said to
\emph{exhibit a pattern $\mathcal{P}$} if $f$ has an invariant set $P$ such that $(T,P,f) \in \mathcal{P}.$

\begin{figure}
\centering
\includegraphics[scale=0.9]{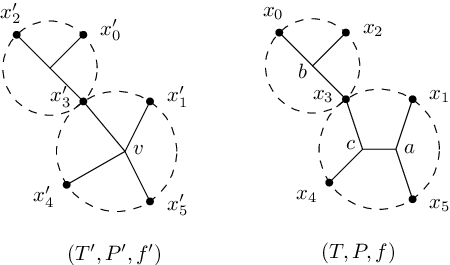}
\caption[fck]{Set $P=\{x_i\}_{i=0}^5$ and $P'=\{x'_i\}_{i=0}^5$. If $\map{f}{T}$ and $\map{f'}{T'}$ are continuous
maps such that $f(x_i)=x_{i+1}$ and $f'(x'_i)=x'_{i+1}$ for $0\leq i\leq 5$, $f(x_5)=x_0$ and $f'(x'_5)=x'_0$,
then the models $(T,P,f)$ and $(T',P',f')$ are equivalent and belong to the same pattern $[T,P,f] = [T',P',f']$.\label{expat}}
\end{figure}

A usual way of measuring the dynamical complexity of a map $\map{f}{X}$ of a compact metric space
is in terms of its \emph{topological entropy}, a notion first introduced in 1965 \cite{AKM}. It is a non-negative real number (or
infinity) that measures how the iterates of the map mix the points of $X$. It will be denoted by $h(f)$. An interval map with
positive entropy is \emph{chaotic} in the sense of Li and Yorke \cite{li-yorke}. The same is true for more general
compact metric spaces \cite{blan}. On the other hand, the dynamics of a map with zero topological entropy is much simpler.

Given a pattern $\mathcal{P}$ in $\mathcal{F}_\mathcal{X}$, we would like to establish, only in terms of the combinatorial
data encoded by $\mathcal{P}$, a lower bound for the dynamical complexity that will be present in any map in $\mathcal{F}_\mathcal{X}$
exhibiting $\mathcal{P}$. In view of what have been said in the previous paragraph, it is natural to define
the \emph{topological entropy of the pattern $\mathcal{P}$}, denoted from now on by $h(\mathcal{P})$, as the infimum of the
topological entropies of all maps in $\mathcal{F}_\mathcal{X}$ exhibiting $\mathcal{P}$.

Although computing the entropy of a continuous map is difficult in general, in some cases the computation of the
entropy of a pattern $\mathcal{P}$ in $\mathcal{F}_\mathcal{X}$ can be easily performed thanks to the existence of
the so called \emph{canonical models}. A \emph{canonical model} of a pattern
$\mathcal{P}$ in $\mathcal{F}_\mathcal{X}$ is a map $f\in\mathcal{F}_\mathcal{X}$ that exhibits $\mathcal{P}$
and satisfies at least the following properties:
\begin{enumerate}
\item[(1)] $f$ is essentially unique and can be constructed from the combinatorial data enclosed in $\mathcal{P}$
\item[(2)] $f$ has minimum entropy in the set of all maps exhibiting $\mathcal{P}$
\item[(3)] the dynamics of $f$ can be completely described using algebraic tools that, in particular, allow us to compute $h(f)$.
\end{enumerate}
From (1--3) it follows that $h(\mathcal{P})$, defined as the infimum of entropies of maps, is in fact a minimum and can be
easily computed as the entropy of the canonical model of $\mathcal{P}$. The existence of canonical
models for patterns has been proved for continuous maps of closed intervals (see \cite{biblia} for a list of
references), homeomorphisms of compact surfaces \cite{fat,Thu} and continuous maps on trees \cite{aglmm}.

Now we are ready to explain the aim of this paper. Several natural questions concerning patterns and entropy arise.
Fix $n\in\mathbb{N}$ and consider the (finite) set of all $n$-periodic tree patterns. An important classification in this set is
given by the zero/positive entropy character of its elements. On the one hand, the zero entropy tree patterns are well understood
and several equivalent characterizations can be found in the literature \cite{Blokh,aglmm,reducibility}. On the other hand,
let $\Pos_n$ be the subset of all $n$-periodic tree patterns with positive entropy. One would like to describe the patterns with
maximal/minimal entropy in $\Pos_n$.

Several advances in the description of the entropy-maximal tree patterns have been reported \cite{ajkm}, but the problem
is still open. In fact, the maximality problem is unsolved even in the particular case of interval patterns \cite{GT,GZ,KS}. Indeed,
the maximal-entropy cyclic permutations of order $n$, when $n$ has the form $4k+2$, are still unknown, although \cite{ajk}
tackles this case from a computational point of view and proposes a conjecture.

In this paper we face the opposite problem: the characterization of the patterns of minimal entropy in
$\Pos_n$. For interval maps, the description of the minimum entropy cycles is known when $n$ is not a
power of two (see \cite{biblia} for a review). In the setting of tree maps and for any $n\ge3$,
an $n$-periodic tree pattern $\mathcal{Q}_n$ was defined in \cite{ajm} that conjecturally has minimal entropy in the set $\Pos_n$
(the problem makes no sense when $n=1,2$, since every periodic pattern of period 1 or 2 has entropy zero), and the
conjecture was proved to be true when $n$ is a power of a prime. See the canonical model of $\mathcal{Q}_n$ in
Figure~\ref{Q6}. The entropy of $\mathcal{Q}_n$ turns out to be $\log(\lambda_n)$, where $\lambda_n$ is the unique
real root of the polynomial $x^n-2x-1$ in $(1,+\infty)$.

\begin{figure}
\centering
\includegraphics[scale=0.8]{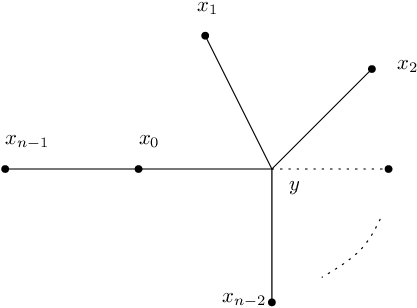}
\caption[fck]{The canonical model $(T,P,f)$ of the pattern $\mathcal{Q}_n$, for
which $P=\{x_i\}_{i=0}^{n-1}$ is time labeled and $f(y)=y$.
\label{Q6}}
\end{figure}

The first main result of this paper states that the conjecture is in fact true for every $n\ge3$.

\begin{MainTheorem}\label{Tachin}
Let $n\ge3$ be a positive integer. Then, $\mathcal{Q}_n$ has minimum entropy in the set $\Pos_n$
of all $n$-periodic patterns with positive entropy. Moreover, $h(\mathcal{P}) > h(\mathcal{Q}_n)=\log(\lambda_n)$ for any
$\mathcal{P}\in\Pos_n$ such that $\mathcal{P}\ne\mathcal{Q}_n$, where $\lambda_n$ is the unique
real root of the polynomial $x^n-2x-1$ in $(1,+\infty)$.
\end{MainTheorem}

Traditionally, reducibility/irreducibility has been another important classification for tree patterns.
A pattern is \emph{reducible} when it has a block structure (see Section~\ref{S3}). Roughly speaking, this
means that the points of the orbit can be partitioned into disjoint subtrees that are permuted under the action
of the map. The notion of reducibility arose early in the study of interval maps and
has been recently extended to the setting of tree patterns \cite{reducibility}. The irreducible tree patterns
are closely related to pseudo-Anosov braid types of periodic orbits of orientation preserving disk homeomorphisms
\cite{fm}. As we will see, every irreducible tree pattern has positive entropy. The dynamic relevance
of the patterns from $\Irr_n$ motivates the study of the minimality of the entropy in this subclass of
$\Pos_n$. For interval maps, the problem was solved in \cite{minor}. Since the minimum entropy pattern
$\mathcal{Q}_n$ turns out to be irreducible, Theorem~\ref{Tachin} incidentally proves that
$\mathcal{Q}_n$ also minimizes the topological entropy in the subclass $\Irr_n$.

\begin{MainCorollary}\label{Tachin2}
Let $n\ge3$ be a positive integer. Then, $\mathcal{Q}_n$ has minimum entropy in the set $\Irr_n$
of all $n$-periodic irreducible patterns. Moreover, $h(\mathcal{P}) > h(\mathcal{Q}_n)=\log(\lambda_n)$ for any
$\mathcal{P}\in\Irr_n$ such that $\mathcal{P}\ne\mathcal{Q}_n$.
\end{MainCorollary}

Now, the problem of determining the
minimum (positive) entropy in the family of all reducible patterns arises. It is not difficult to see that
$\Pos_n\setminus\Irr_n\ne\emptyset$ if and only if $n$ is not a prime and $n\ge6$. By Theorem~\ref{Tachin},
the minimum positive entropy for any reducible pattern is strictly larger than $\log(\lambda_n)$. The
second main result of this paper gives the minimum entropy in $\Pos_n\setminus\Irr_n$.
In this case, however, the minimum entropy pattern is not unique.

\begin{MainTheorem}\label{Tachin3}
Let $n\ge6$ be a composite number. Then, the minimum positive entropy in the set of all reducible $n$-periodic
patterns is $\log(\lambda_{n/p})/p$, where $p$ is the smallest prime factor of $n$.
\end{MainTheorem}

This paper is organized as follows. In Section~\ref{S2} we introduce formally the basic notions of
pattern, canonical model and path transition matrix, and recall how to compute the topological entropy
of a pattern. In Section~\ref{S3} we review some classic notions and results about block structures and reducibility
for tree patterns, that we use in Section~\ref{S6} to recall the characterization of zero entropy periodic patterns.
A deeper study of the structure of zero entropy paterns is carried out in Section~\ref{S6b}.
In Section~\ref{S4} we briefly recall a mechanism, first introduced in \cite{ajm}, that allows us to compare the
entropies of two patterns $\mathcal{P}$ and $\mathcal{O}$ when $\mathcal{O}$ has been obtained by joining together
several discrete components of $\mathcal{P}$. Section~\ref{S5} is devoted to the task of explaining the strategy
of the proof of Theorem~\ref{Tachin}. As we will see, the proof is by induction on the period $n$ and relies on
a core result, Theorem~\ref{2inners_3op}, that is stated in the same section and proved in Section~\ref{S7} using
the results of Section~\ref{S6b}. The use of this result allows us to prove Theorem~\ref{Tachin} for almost all patterns,
with two particular exceptions: the \emph{$k$-flowers} (patterns with $k$ discrete components attached at a unique central
point) and the \emph{triple chain}, a pattern with three consecutive discrete components. We deal with these two cases
in Sections~\ref{S8} and \ref{S9} respectively. Putting all together, we prove Theorem~\ref{Tachin} in Section~\ref{S10}.
Finally, Section~\ref{S11} is devoted to the proof of Corollary~\ref{Tachin2} and Theorem~\ref{Tachin3}.

\section{Patterns and canonical models}\label{S2}
In this section we formalize the definitions outlined in the Introduction. We also recall how to compute
the topological entropy of a pattern by using purely combinatorial tools. Finally we define the pattern that
will be proved to have minimum positive entropy.

A \emph{tree} is a compact uniquely arcwise connected space which is a
point or a union of a finite number of intervals (by an \emph{interval} we mean any space homeomorphic to $[0,1]$).
Any continuous map $\map{f}{T}$ from a tree $T$ into itself will be called a \emph{tree map}. A set $X\subset T$
is said to be \emph{$f$-invariant} if $f(X)\subset X$. For each $x\in T$, we define the \emph{valence} of $x$ to
be the number of connected components of $T\setminus\{x\}$. A point of valence different from 2 will be called
a \emph{vertex} of $T$ and the set of vertices of $T$ will be denoted by $V(T)$. Each point of valence 1 will
be called an \emph{endpoint} of $T$. The set of such points will be denoted by $\En(T)$. Also, the closure of
a connected component of $T \setminus V(T)$ will be called an \emph{edge of $T$}.

Given any subset $X$ of a topological space, we will denote by $\Int(X)$ and $\Cl(X)$ the interior and the closure
of $X$, respectively. For a finite set $P$ we will denote its cardinality by $|P|$.

A triplet $(T,P,f)$ will be called a \emph{model} if $\map{f}{T}$ is a tree map and $P$ is a finite $f$-invariant set
such that $\En(T)\subset P$. In particular, if $P$ is a periodic orbit of $f$ and $|P|=n$ then $(T,P,f)$ will be called
an \emph{$n$-periodic model}. Given $X\subset T$ we will define the \emph{connected hull} of $X$, denoted by
$\chull{X}_T$ or simply by $\chull{X}$, as the smallest closed connected subset of $T$ containing $X$. When
$X=\{x,y\}$ we will write $[x,y]$ to denote $\chull{X}$. The notations $(x,y)$, $(x,y]$ and $[x,y)$ will be understood
in the natural way.

An $n$-periodic orbit $P=\{x_i\}_{i=0}^{n-1}$ of a map $\theta$ will be said to be \emph{time labeled} if
$\theta(x_i)=x_{i+1}$ for $0\le i<n-1$ and $\theta(x_{n-1})=x_0$.

Let $T$ be a tree and let $P\subset T$ be a finite subset of $T$. The pair $(T,P)$ will be called a \emph{pointed tree}.
Two points $x,y$ of $P$ will be said to be \emph{consecutive} if $(x,y)\cap P=\emptyset$. Any maximal subset of $P$
consisting only of pairwise consecutive points will be called a \emph{discrete component} of $(T,P)$. We say that two
pointed trees $(T,P)$ and $(T',P')$ are \emph{equivalent} if there exists a bijection $\map{\phi}{P}[P']$ which preserves
discrete components. The equivalence class of a pointed tree $(T,P)$ will be denoted by $[T,P]$.

Let $(T,P)$ and $(T',P')$ be equivalent pointed trees, and let
$\map{\theta}{P}$ and $\map{\theta'}{P'}$ be maps. We will say that
$\theta$ and $\theta'$ are \emph{equivalent} if $\theta'=
\phi\circ\theta \circ\phi^{-1}$ for a bijection
$\map{\phi}{P}[P']$ which preserves discrete components. The
equivalence class of $\theta$ by this relation will be denoted by
$[\theta]$. If $[T,P]$ is an equivalence class of pointed trees and
$[\theta]$ is an equivalence class of maps then the pair
$([T,P],[\theta])$ will be called a \emph{pattern}.
We will say that a model $(T,P,f)$ \emph{exhibits} a pattern
$(\mathcal{T},\Theta)$ if $\mathcal{T}=[\chull{P}_T,P]$ and
$\Theta=[f\evalat{_{P}}]$.

Despite the fact that the notion of a discrete component is defined for pointed trees,
by abuse of language we will use the expression \emph{discrete component of a pattern},
which will be understood in the natural way since the number of discrete components
and their relative positions are the same for all models of the pattern.

Recall that the topological entropy of a continuous tree map $f$ is denoted by $h(f)$.
Given a pattern $\mathcal{P}$, the topological entropy of $\mathcal{P}$ is defined to be
\[
h(\mathcal{P}) := \inf \set{h(f)}{ (T,P,f)\ \text{is a model exhibiting}\
\mathcal{P} }.
\]

The simplest models exhibiting a given pattern are the monotone ones,
defined as follows. Let $\map{f}{T}$ be a tree map map. Given $a,b\in T$ we say that
$f\evalat{[a,b]}$ is \emph{monotone} if $f([a,b])$ is either an
interval or a point and $f\evalat{[a,b]}$ is monotone as an interval
map. Let $(T,P,f)$ be a model. A pair $\{a,b\}\subset P$
will be called a \emph{basic path of $(T,P)$} if it is contained in a
single discrete component of $(T,P)$. We will say that $f$ is
\emph{$P$-monotone} if $f\evalat{[a,b]}$ is
monotone for any basic path $\{a,b\}$. The model $(T,P,f)$ will then be said
to be \emph{monotone}. In such case, Proposition~4.2 of \cite{aglmm}
states that the set $P\cup V(T)$ is $f$-invariant (recall that $V(T)$ stands for
the set of vertices of $T$). Hence, the map $f$ is also $(P\cup V(T))$-monotone.
Observe that the notion of $P$-monotonicity is much more restrictive than the usual
topological notion of a \emph{monotone map} (full preimages of continua are continua).

Theorem~A of \cite{aglmm} states that every pattern $\mathcal{P}$ has monotone models, and that
for every monotone model $(T,P,f)$ of $\mathcal{P}$, $h(f)=h(\mathcal{P})$. Moreover,
there exists a special class of monotone models, satisfying several extra properties
that we omit here, called \emph{canonical models}. Theorem~B of \cite{aglmm} states
that every pattern has a canonical model. Moreover, given two canonical models
$(T,P,f)$ and $(T',P',f')$ of the same pattern there exists a
homeomorphism $\map{\phi}{T}[T']$ such that $\phi(P) = P'$ and
$f' \circ \phi\evalat{P} = \phi \circ f\evalat{P}$. Hence, the
canonical model of a pattern is essentially unique. Summarizing, we have
the following result.

\begin{theorem}\label{A-AGLMM}
Let $\mathcal{P}$ be a pattern. Then the following statements hold.
\begin{enumerate}
\item There exists a canonical model of $\mathcal{P}$.
\item The canonical model $(T,P,f)$ of $\mathcal{P}$ satisfies $h(f)=h(\mathcal{P})$.
\end{enumerate}
\end{theorem}

It is worth noticing that the proof of Theorem~\ref{A-AGLMM} gives a finite algorithm to
construct the canonical model of any pattern.
For instance, the model $(T,P,f)$ in the right picture of Figure~\ref{expat} is the canonical model of the
corresponding pattern. The $P$-monotonicity of $f$ determines that $f(a) = b,$ $f(b) = c,$ and $f(c) = c.$
Observe also that the left model $(T',P',f')$ of Figure~\ref{expat}, a representative of the same pattern,
cannot be $P'$-monotone, since in this case we would have $f'(v) \in f'([x'_2,x'_6]) \cap f'([x'_4,x'_5]) =
[x'_3,x'_1] \cap [x'_5,x'_6] = \emptyset.$

There is a combinatorial procedure to compute the entropy of a pattern $\mathcal{P}$ which
does not require the construction of its canonical model. Indeed, $h(\mathcal{P})$ can be
obtained from the transition matrix of a combinatorial directed graph that can be derived
independently of the images of the vertices in any particular monotone model of
the pattern. Let us recall this procedure.

A \emph{combinatorial directed graph} is a pair $\mathcal{G}=(V,U)$
where $V = \{v_1,v_2,\dots,v_k\}$ is a finite set and $U\subset V\times V$.
The elements of $V$ are called the \emph{vertices} of $\mathcal{G}$ and
each element $(v_i,v_j)$ in $U$ is called an \emph{arrow} (from $v_i$ to
$v_j$) in $\mathcal{G}$. Such an arrow is usually denoted by $v_i\rightarrow v_j$.
The notions of \emph{path} and \emph{loop} in $\mathcal{G}$ are defined as usual.
The \emph{length} of a path is defined as the number of arrows in the path.
The \emph{transition matrix} of $\mathcal{G}$ is a $k\times k$ binary matrix
$(m_{ij})_{i,j=1}^k$ such that $m_{ij}=1$ if and only if there is an arrow
from $v_i$ to $v_j$, and $m_{ij}=0$ otherwise.

Let $\{\pi_1,\pi_2,\ldots,\pi_k\}$ be the set
of basic paths of the pointed tree $(T,P)$. We will say that $\pi_i$ \emph{$f$-covers} $\pi_j$,
denoted by $\pi_i\rightarrow\pi_j$, whenever $\pi_j\subset \chull{f(\pi_i)}_T$. The
\emph{$\mathcal{P}$-path graph} is the combinatorial directed graph whose vertices are in one-to-one
correspondence with the basic paths of $(T,P)$, and there is an arrow from the vertex $i$ to the vertex
$j$ if and only if $\pi_i$ $f$-covers $\pi_j$. The associated transition matrix, denoted by $M_\mathcal{P}$,
will be called the \emph{path transition matrix of $\mathcal{P}$}. It can be seen that the definitions of
the $\mathcal{P}$-path graph and the matrix $M_{\mathcal{P}}$ are independent of the particular
choice of the model $(T,P,f)$. Thus, they are well-defined pattern invariants.

For any square matrix $M$, we will denote its \emph{spectral radius}
by $\rho(M)$. We recall that it is defined as the maximum of the
moduli of the eigenvalues of $M$.

\begin{remark}\label{patent}
Let $M_{\mathcal{P}}$ be the path transition matrix of a pattern $\mathcal{P}$. Then (see \cite{aglmm}), the topological
entropy of $\mathcal{P}$ can be computed as $h(\mathcal{P})=\log\max\{\rho(M_{\mathcal{P}}),1\}$.
\end{remark}

\label{qn}
To end this section we define the patterns that will be showed to have minimum positive entropy. Let $n\in\mathbb{N}$
with $n\ge3$. Let $\mathcal{Q}_n$ be the $n$-periodic pattern $([T,P],[\theta])$ such that $P=\{x_0,x_1,\ldots,x_{n-1}\}$
is time labeled and $(T,P)$ has two discrete components, $\{x_{n-1},x_0\}$ and $\{x_0,x_1,\ldots,x_{n-2}\}$. In Figure~\ref{Q6}
we show the canonical model of $\mathcal{Q}_n$. Observe that $\mathcal{Q}_3$ is nothing but the
3-periodic \v{S}tefan cycle of the interval \cite{ste}. In \cite{ajm} the authors prove that $h(\mathcal{Q}_n)=\log(\lambda_n)$, where
$\lambda_n$ is the unique real root of the polynomial $x^n-2x-1$ in $(1,+\infty)$. We will use the following properties of the
numbers $\lambda_n$. Statement (a) is proved in Proposition~3.1 of \cite{ajm}, while statement (b) is an easy exercise.

\begin{proposition}\label{propietats}
Let $n$ be any positive integer with $n\ge 3$. Then:
\begin{enumerate}
\item $\lambda_{n+1}<\lambda_n$
\item $\sqrt[n]{4}>\lambda_n$.
\end{enumerate}
\end{proposition}

\section{Block structures, skeletons and $\pi$-reducibility}\label{S3}
The zero entropy tree patterns will play a central role in this paper. The characterization
of such patterns was first given in \cite{aglmm}, and another description was proven
to be equivalent in \cite{reducibility}. We will use this second approach, and this section is devoted
to recall the necessary notions and results. The characterization of zero entropy periodic patterns
relies on the notion of \emph{block structure}, that is classic in the field of Combinatorial Dynamics. In
the literature one can find several kinds of block structures and related notions for periodic orbits.
In the interval case, the Sharkovskii's \emph{square root construction} \cite{shar} is an early example
of a block structure. The notion of \emph{extension}, first appeared in \cite{block2}, gives
rise to some particular cases of block structures. Also the notion of \emph{division},
introduced in \cite{lmpy} for interval periodic orbits and generalized in \cite{ay1} in order to study the
entropy and the set of periods for tree maps, is a particular case of block structure.

\begin{remark}\label{standing1}
All patterns considered in this paper will be periodic. Given an $n$-periodic pattern $\mathcal{P}$, by abuse
of language we will speak about the \emph{points} of $\mathcal{P}$, and by default we will consider that such points
are time labeled with the integers $\{0,1,\ldots,n-1\}$. Often we will identify a point in $\mathcal{P}$ with its
time label. In agreement with such conventions, the points of the patterns shown in the pictures will be
simply integers in the range $[0,n-1]$. See for instance Figure~\ref{ull2blocks}.
\end{remark}

A pattern will be said to be \emph{trivial} if it has only one discrete component. It is easy
to see that the entropy of any trivial pattern is zero.

\label{bs}
Let $\mathcal{P}=([T,P],[f])$ be a nontrivial $n$-periodic pattern with $n\ge3$. For $n>p\ge2$, we will say that
$\mathcal{P}$ \emph{has a $p$-block structure} if there exists a partition $P=P_0\cup
P_1\cup\ldots\cup P_{p-1}$ such that $f(P_i)=P_{i+1\bmod p}$ for $i\ge0$,
and $\chull{P_i}_T\cap P_j=\emptyset$ for $i\ne j$. In this case, $p$ is a strict divisor of $n$ and
$|P_i|=n/p$ for $0\le i<p$. The sets $P_i$ will be called \emph{blocks}, and the blocks will be said to be
\emph{trivial} if each $P_i$ is contained in a single discrete component of $\mathcal{P}$ (equivalently, each pattern
$([\chull{P_i}_T,P_i],[f^p])$ is trivial). Note that
$\mathcal{P}$ can have several block structures, but only one $p$-block structure for any given divisor $p$ of $n$.
If $\mathcal{P}$ has structures of trivial blocks, the one with blocks with maximum cardinality will be called
a \emph{maximal structure}.

From the equivalence relation which defines the class of models belonging to the pattern $\mathcal{P}$ it easily follows
that the notions defined in the previous paragraph do not depend on the particular model $(T,P,f)$ representing $\mathcal{P}$.

\begin{remark}[{\bf Standing convention}]\label{standing2}
Let $\mathcal{P}$ be an $n$-periodic pattern whose points are time labeled as $\{0,1,\ldots,n-1\}$. When $\mathcal{P}$ has
a block structure of $p$ blocks $P_0\cup P_1\cup\ldots\cup P_{p-1}$, by convention we will always assume that the time labels
of the blocks have been chosen in such a way that $0\in P_0$.
\end{remark}

Let $(T,P,f)$ be the canonical model of $\mathcal{P}$. A $p$-block structure
$P_0\cup P_1\cup\ldots\cup P_{p-1}$ for $\mathcal{P}$ will be said to be \emph{separated} if $\chull{P_i}_T\cap\chull{P_j}_T=\emptyset$
for $i\ne j$. Note that the separability of a block structure for a pattern depends on the particular topology of its canonical model
and, in consequence, cannot be determined directly from the combinatorial data of $\mathcal{P}$ a priori. However, recall that the canonical
model of a pattern $\mathcal{P}$ is unique and can be algorithmically computed from $\mathcal{P}$. So, this is an intrinsic notion.

In Figure~\ref{ull2blocks} we show an example of a 8-periodic pattern $\mathcal{P}$ admitting
both a 4-block structure given by $P_0=\{0,4\}$, $P_1=\{1,5\}$, $P_2=\{2,6\}$, $P_3=\{3,7\}$
and a 2-structure given by $Q_0=\{0,2,4,6\}$, $Q_1=\{1,3,5,7\}$. Note that in both cases the blocks are trivial,
and $Q_0\cup Q_1$ is a maximal structure by definition. As it has been said, one can determine these block structures directly
in the combinatorial representation of $\mathcal{P}$, without checking any particular topology. See Figure~\ref{ull2blocks} (left).
On the contrary, to determine the separability of a block structure one has to construct the canonical model of $\mathcal{P}$,
which is shown in the same figure (right). Here we see that $Q_0\cup Q_1$ is separated, while $P_0\cup P_1\cup P_2\cup P_3$
is not (the convex hulls of the blocks $P_0$ and $P_2$, which are respectively the intervals $[0,4]$ and $[2,6]$, intersect at the vertex $a$).

\begin{figure}
\centering
\includegraphics[scale=0.65]{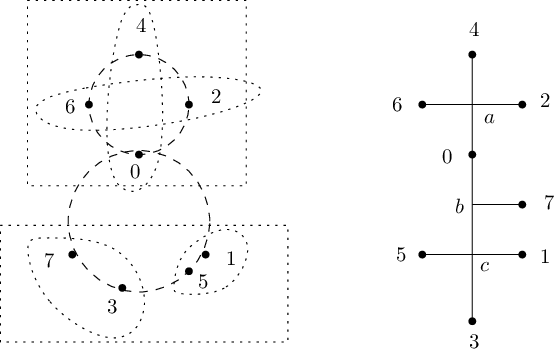}
\caption[fck]{Left: an 8-periodic pattern $\mathcal{P}$ admitting two block structures with trivial blocks.
Right: the canonical model $(T,P,f)$ of $\mathcal{P}$, for which the images of the vertices are $f(a)=c$,
$f(b)=0$ and $f(c)=a$.\label{ull2blocks}}
\end{figure}

Let $\mathcal{P}$ be an $n$-periodic pattern and let $(T,P,f)$ be the
canonical model of $\mathcal{P}$. Let $P=P_0\cup P_1\cup\ldots\cup
P_{p-1}$ be a separated $p$-block structure for $\mathcal{P}$. Then,
$f(\chull{P_i})=\chull{P_{i+1\bmod p}}$. The \emph{skeleton of $\mathcal{P}$}
(associated to this block structure) is a $p$-periodic pattern $\mathcal{S}$ defined as follows.
Consider the tree $S$ obtained from $T$ by collapsing each tree
$\chull{P_i}$ to a point $x_i$. Let $\map{\kappa}{T}[S]$ be the
standard projection, which is bijective on $T\setminus\cup_i
\chull{P_i}$ and satisfies $\kappa(\chull{P_i})=x_i$. Set
$Q=\kappa(P)=\{x_0,x_1,\ldots,x_{p-1}\}$ and define $\map{\theta}{Q}$ by
$\theta(x_i)=x_{i+1\bmod p}$. Then the \emph{skeleton} $\mathcal{S}$ of $\mathcal{P}$
is defined to be the $p$-periodic pattern $([S,Q],[\theta])$.

\begin{remark}[{\bf Standing convention}]\label{standing3}
Let $\mathcal{P}$ be an $n$-periodic pattern whose points are time labeled as $\{0,1,\ldots,n-1\}$. Assume that
$\mathcal{P}$ has a separated $p$-block structure. From the convention established in Remark~\ref{standing2}, each point
of $\mathcal{P}$ labeled as $i$ belongs to the block $P_{i\bmod{p}}$. From now on we adopt the convention that the
$p$ points of the skeleton have time labels $\{0,1,\ldots,p-1\}$ such that the point $i$ of the skeleton corresponds
to the collapse of the block $P_i$.
\end{remark}

\begin{example}\label{exesk}
Let us see an example of construction of the skeleton. Consider the 8-periodic pattern $\mathcal{P}$ consisting of
two discrete components $\{0,2,6\}$, $\{0,1,3,4,5,7\}$ (Figure~\ref{topcol}, left). Then, $P_0=\{0,4\}$, $P_1=\{1,5\}$,
$P_2=\{2,6\}$, $P_3=\{3,7\}$ defines a structure of 4 trivial blocks. By checking the canonical model $(T,P,f)$, which is
shown in Figure~\ref{topcol} (center), we see that $\chull{P_i}_T\cap\chull{P_j}_T=\emptyset$ when $i\ne j$. Thus, the
structure is separated. The corresponding skeleton is obtained by collapsing the convex hull of each block
to a point, giving the 4-periodic pattern $\mathcal{S}$ shown in Figure~\ref{topcol} (right).
\end{example}

\begin{figure}
\centering
\includegraphics[scale=0.65]{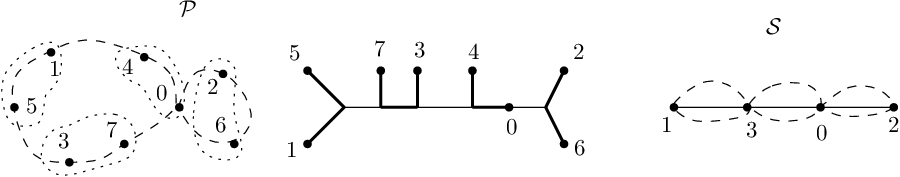}
\caption[fck]{Left: an 8-periodic pattern $\mathcal{P}$ with a separated structure of 4 trivial blocks. Center:
the canonical model $(T,P,f)$ of $\mathcal{P}$, the convex hulls of the blocks marked with thick lines.
Right: the corresponding skeleton. \label{topcol}}
\end{figure}

The entropies of a pattern $\mathcal{P}$ with a separated structure of trivial blocks and its associated
skeleton coincide, as the following result (a reformulation of Proposition~8.1 of \cite{aglmm}) states.

\begin{proposition}\label{81}
Let $\mathcal{P}$ be a pattern with a separated structure of trivial blocks. Let $\mathcal{S}$ be the
corresponding skeleton. Then, $h(\mathcal{S})=h(\mathcal{P})$.
\end{proposition}

Going back to Example~\ref{exesk}, note that the obtained skeleton $\mathcal{S}$ is a zero entropy
interval pattern. Then, $h(\mathcal{P})=0$ by Proposition~\ref{81}.

As a consequence of Proposition~\ref{81} we have the following result, that will be used in the proof
of the main theorem of this paper.

\begin{corollary}\label{skeletonok}
Let $\mathcal{P}$ an $n$-periodic pattern with a separated structure of $p$ trivial blocks. Let
$\mathcal{S}$ be the corresponding skeleton. If $h(\mathcal{S})\ge\log(\lambda_p)$, then $h(\mathcal{P})>\log(\lambda_n)$.
\end{corollary}
\begin{proof}
Since $p$ is a strict divisor of $n$, it is a direct consequence of Propositions~\ref{81} and \ref{propietats}(a).
\end{proof}

The existence of a separated structure of trivial blocks for a pattern $\mathcal{P}$ has a strong connection with
the path transition matrix of $\mathcal{P}$, via the iterative behaviour of some particular basic paths of $\mathcal{P}$.
Let us explain it. Let $\mathcal{P}$ be a periodic pattern and let $\pi$ be a basic path of $\mathcal{P}$. Consider any
model $(T,P,f)$ of $\mathcal{P}$. For $k\ge1$,
we will say that $\pi$ \emph{splits in $k$ iterates} \label{pagesplit} if $f^i(\pi)$ is a basic path of $\mathcal{P}$ for $0\le i<k$ and $f^k(\pi)$
is not a basic path of $\mathcal{P}$. Equivalently, $f^i(\pi)$ only $f$-covers
$f^{i+1}(\pi)$ for $0\le i<k$ and $f^{k-1}(\pi)$ $f$-covers at least two different basic paths. We say that a basic path
$\pi$ \emph{never splits} if $f^i(\pi)$ is a basic path for every $i\ge0$. In this case, we will say that $\mathcal{P}$
is \emph{$\pi$-reducible}. As an example, the path $\pi=\{0,4\}$ for the pattern $\mathcal{P}$ in Figure~\ref{topcol}
never splits, so $\mathcal{P}$ is $\pi$-reducible. On the other hand, let $\sigma$ be the path $\{4,7\}$ on the same pattern. Note that
$f(\sigma)=\{5,0\}$ is a basic path, while $f^2(\sigma)=\{6,1\}$ is not. Then, $\sigma$ splits in 2 iterates and $f^2$-covers
the two basic paths $\{6,0\}$ and $\{0,1\}$.

The $\pi$-reducibility of a pattern with respect to a basic path $\pi$ is equivalent to
the existence of a separated structure of trivial blocks, as the following result states.

\begin{proposition}\label{twocharact}
Let $\mathcal{P}$ be a periodic pattern. Then, $\mathcal{P}$ is $\pi$-reducible for a basic path $\pi$ if and only if
$\mathcal{P}$ has a maximal and separated structure of trivial blocks. In this case, $\mathcal{P}$ is $\sigma$-reducible for any
basic path $\sigma$ contained in a block.
\end{proposition}
\begin{proof}
The `only if' part of the first statement is Proposition~9.5 of \cite{ajm}, while its `if' part and the second claim
easily follow from the definition of a trivial block structure.
\end{proof}

\section{A mechanism to compare entropies}\label{S4}
Another key ingredient to prove Theorem~\ref{Tachin} is a tool, first introduced in \cite{ajm}, that allows us to compare the entropies of two patterns $\mathcal{P}$ and $\mathcal{O}$ when $\mathcal{O}$ has been obtained by joining together several discrete components of $\mathcal{P}$. For the sake of brevity, here we will give a somewhat informal (though completely clear) version of this procedure.

Let $(T,P,f)$ be a model of a pattern $\mathcal{P}$. We recall that two discrete components of $(T,P)$ are either disjoint or intersect at a single point of $P$. Two discrete
components $A,B$ of $(T,P)$ will be said to be \emph{adjacent at $x\in P$} (or simply \emph{adjacent}) if $A\cap B=\{x\}$. A point $z\in P$ will be said to be \emph{inner}
if $z$ belongs to $k\ge2$ discrete components of $(T,P)$, all being pairwise adjacent at $z$.

Now let $x\in P$ be an inner point and let $A,B$ be two discrete components adjacent at $x$. If we join together $A$ and $B$ to get a new discrete component $A\cup B$ and keep intact the remaining components, we get a new pattern $\mathcal{O}$. We will say that $\mathcal{O}$ is an \emph{opening of $\mathcal{P}$} (with respect to the inner point $x$ and the discrete components $A$ and $B$). As an example, see Figure~\ref{brutty2}, where $\mathcal{O}$ is an opening of $\mathcal{P}$ with respect to the inner point 5 and the discrete components $A=\{2,5,6\}$ and $B=\{0,5\}$, while $\mathcal{R}$ is an opening of $\mathcal{P}$ with respect to the inner point 5 and the discrete components $B$ and $C=\{1,3,5\}$.

\begin{figure}
\centering
\includegraphics[scale=0.65]{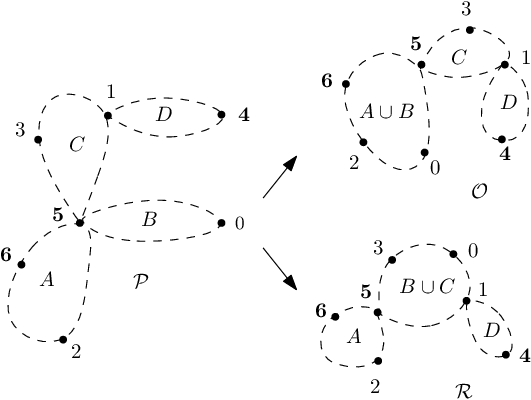}
\caption[fck]{Two different openings of $\mathcal{P}$.\label{brutty2}}
\end{figure}

\begin{remark}[{\bf Standing convention}]\label{standing4}
As it is clear from the examples shown in Figure~\ref{brutty2}, we are implicitly assuming that the labeling of the points
of an $n$-periodic pattern $\mathcal{P}$ fixes the labeling of the points of any opening of $\mathcal{P}$.
\end{remark}

As one may expect from intuition, the entropy of a model decreases when performing an opening, as the following result (Theorem~5.3 of \cite{ajm}) states.
\begin{theorem}\label{5.3}
Let $\mathcal{P}$ and $\mathcal{O}$ be $n$-periodic patterns. If $\mathcal{O}$ is an opening of $\mathcal{P}$, then $h(\mathcal{P})\ge h(\mathcal{O})$.
\end{theorem}

We finish this section stating that the property for a pattern of having a block structure is preserved by openings.
The result is a direct consequence of the definition of a block structure and the fact that no new inner points are created
after performing an opening.

\begin{lemma}\label{openingblock}
Let $\mathcal{P}$ be a periodic pattern with a block structure and let $\mathcal{O}$ be an opening of $\mathcal{P}$. Then,
$\mathcal{O}$ has a block structure.
\end{lemma}

\section{Strategy of the proof of Theorem~\ref{Tachin}}\label{S5}
In this section we give a general overview of the proof of Theorem~\ref{Tachin}, in order
to justify the need for the several techniques and results deployed in the subsequent sections.

We will prove Theorem~\ref{Tachin} by induction on the period $n$. So, assume that we have
an $n$-periodic pattern $\mathcal{P}$ and that the result is true for every pattern with period less than $n$.

The first step is a simplification process based on the opening mechanism. Recall (Theorem~\ref{5.3}) that
after performing an opening on $\mathcal{P}$, the entropy $h$ of the obtained pattern is less or equal to
$h(\mathcal{P})$. If $h$ is still positive, we can perform again an opening and so on, until we get
a pattern with positive entropy such that every new opening leads to entropy zero. In other words,
we can assume that $\mathcal{P}$ satisfies the following property:
\begin{equation}\label{opening0}
\mbox{Every opening of $\mathcal{P}$ is a zero entropy pattern}.\tag{$\star$}
\end{equation}\label{pagina}
Property (\ref{opening0}) is very restrictive and has a strong consequence: a pattern
satisfying (\ref{opening0}) is, \emph{generically}, $\pi$-reducible. More precisely,
we have the following result, that will be proved in Section~\ref{S7}.

\begin{MainTheorem}\label{2inners_3op}
Let $\mathcal{P}$ be an $n$-periodic pattern with positive entropy such that any opening of $\mathcal{P}$ has entropy zero.
Assume that $\mathcal{P}$ has at least two inner points and at least three openings. Then, $\mathcal{P}$ is $\pi$-reducible
for some basic path $\pi$.
\end{MainTheorem}

If $\mathcal{P}$ satisfies the hypothesis of Theorem~\ref{2inners_3op}, then it is $\pi$-reducible. So, we can consider its
skeleton $\mathcal{S}$, with the same entropy but with a period that strictly divides $n$, and use the induction hypothesis.

\begin{figure}
\centering
\includegraphics[scale=0.65]{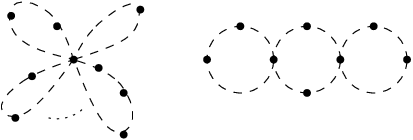}
\caption[fck]{A $k$-flower (left) and a triple chain (right).\label{trefoil}}
\end{figure}

The above argument is the core idea of the proof of Theorem~\ref{Tachin}, but we are left with two special cases for which
we cannot assure that property (\ref{opening0}) implies $\pi$-reducibility: the \emph{$k$-flowers} and the
\emph{triple chain}. A \emph{$k$-flower} is a pattern consisting on $k\ge2$ discrete components (the \emph{petals})
attached at a unique inner point. A pattern having three discrete components and two inner points
will be called a \emph{triple chain}. See Figure~\ref{trefoil}. The reader will find easy to convince that the flowers
and the triple chain are the two sort of patterns that do not satisfy the property of having at least two inner points
and at least three openings.

The cases of the $k$-flowers and the triple chain will be tackled in Sections~\ref{S8} and \ref{S9} respectively. Concerning
the $k$-flowers, the case $k=2$ is specially simple since Theorem~\ref{Tachin} follows directly from a previous result in \cite{forward}.
On the other hand, for $k\ge3$ we construct an $n'$-periodic pattern, where $n'$ is a strict divisor of $n$, whose entropy can be
put in relation with that of $\mathcal{P}$, and then we use the induction hypothesis. Finally, in the case of the triple chain
we compute directly lower bounds of the entropy by counting coverings in the $\mathcal{P}$-path graph (equivalently, entries in the
path transition matrix).

\section{Structure of zero entropy patterns}\label{S6}
Although a point is an element of a topological space and a pattern is a combinatorial object defined as an equivalence class of pointed trees,
recall that by abuse of language we talk about the \emph{points} of a pattern. The same translation from topology to combinatorics can be
applied to the terms \emph{valence}, \emph{inner point} and \emph{endpoint}. The (combinatorial) \emph{valence} of a point $x$ of a pattern $\mathcal{P}$
is defined as the number of discrete components of $\mathcal{P}$ containing $x$. Recall that an \emph{inner point of $\mathcal{P}$} has been
defined as a point of combinatorial valence larger than 1. Otherwise, the point will be called an \emph{endpoint of $\mathcal{P}$}. Let $x$ be a point of
$\mathcal{P}$ of combinatorial valence $\nu$. Obviously, for any model $(T,P,f)$ of $\mathcal{P}$, the (topological) valence of the point of $T$ corresponding
to $x$ is the same and equals $\nu$. In consequence, $x$ is an endpoint (respectively, an inner point) of $\mathcal{P}$ if and only if the point corresponding
to $x$ in any model $(T,P,f)$ is an endpoint (respectively, a point of valence larger than 1) of the tree $T$. So, in what follows we will drop the words
\emph{combinatorial} and \emph{topological} and will use these terms indistinctly in both senses.

The strategy outlined in Section~\ref{S5} relies strongly in using property (\ref{opening0}), that depends on
the notion of \emph{zero entropy pattern}. So, we start this section with the following recursive characterization of zero
entropy patterns, that uses the notions of block structure and skeleton presented in Section~\ref{S3}. It is Proposition~5.6
of \cite{reducibility}.

\begin{proposition}\label{56}
Let $\mathcal{P}$ be an $n$-periodic pattern. Then, $h(\mathcal{P})=0$ if and only if either $\mathcal{P}$
is trivial or has a maximal separated structure of trivial blocks such that the associated skeleton has entropy $0$.
\end{proposition}

\begin{figure}
\centering
\includegraphics[scale=0.65]{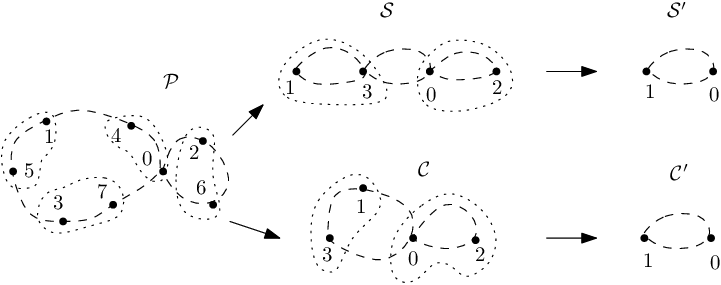}
\caption[fck]{Top: a sequence of skeletons. Bottom: the sequence of combinatorial collapses according to
Definition~\ref{explosions}.\label{comcol}}
\end{figure}

Obviously we can use Proposition~\ref{56} recursively, in the sense that the skeleton $\mathcal{S}$, with entropy
zero and a period that strictly divides that of $\mathcal{P}$, has also a maximal separated structure of trivial blocks
with an associated skeleton $\mathcal{S}'$ of entropy zero. We can thus iterate the process as many times as necessary
to finally obtain a trivial pattern. Consider, for instance, the zero entropy pattern $\mathcal{P}$ of Example~\ref{exesk},
whose skeleton $\mathcal{S}$ was shown in Figure~\ref{topcol}. This skeleton has a maximal separated structure of 2 trivial blocks,
with the associated skeleton $\mathcal{S}'$ being a trivial pattern of 2 points. See the complete sequence of skeletons
in Figure~\ref{comcol} (top). Note that the previous simplification process cannot be carried out without checking the
particular topology of the involved canonical models. Indeed, if we ignore the topology of the tree $T$ in the canonical
model $(T,P,f)$ of $\mathcal{P}$ (that is shown in Figure~\ref{topcol}), for the skeleton it is not possible to decide,
only from the combinatorics of $\mathcal{P}$, between the patterns $\mathcal{S}$ and $\mathcal{C}$ depicted in
Figure~\ref{comcol}. To overcome this dependence from the topology, next we propose a similar but purely combinatorial
simplification mechanism over zero entropy patterns.

\begin{defi}\label{combicolla}
Let $\mathcal{P}=([T,P],[f])$ be a zero entropy $n$-periodic pattern. Let $P_0\cup P_1\cup\ldots\cup P_{p-1}$ be the maximal
and separated structure of trivial blocks given by Proposition~\ref{56}. A $p$-periodic pattern $\mathcal{C}=([S,Q],[g])$ will
be called the \emph{combinatorial collapse of $\mathcal{P}$} if the following properties are satisfied:
\begin{enumerate}
\item[(a)] $g(i)=j$ if and only if $f(P_i)=P_j$
\item[(b)] For any $0\le i<j\le p-1$, there is a discrete component of $\mathcal{P}$ intersecting the blocks $P_i,P_j$ if
and only if there is a discrete component of $\mathcal{C}$ containing the points $i,j$.
\end{enumerate}
We will say that the point $i$ of $\mathcal{C}$ is the \emph{collapse} of the block $P_i$ of $\mathcal{P}$.
Property (a) above implies that the standing convention established in Remark~\ref{standing3} about the labeling of the
points of a skeleton translates verbatim to the labeling of the points of a combinatorial collapse.
\end{defi}

Note that, by definition, the combinatorial collapse is unique, since it is always carried out over the maximal structure of trivial blocks.

As an example, the pattern $\mathcal{C}$ shown in Figure~\ref{comcol} (bottom) is the combinatorial collapse of $\mathcal{P}$.
Note that the skeleton $\mathcal{S}$ does not satisfy property (b) of Definition~\ref{combicolla}: the blocks
$P_0=\{0,4\}$ and $P_1=\{1,5\}$ intersect a single discrete component in $\mathcal{P}$, while the corresponding points
$0,1$ of $\mathcal{S}$ are contained in different discrete components.

Notice that, if $\mathcal{P}$ is a zero entropy pattern, then the combinatorial collapse $\mathcal{C}$ of $\mathcal{P}$
can be obtained from the skeleton $\mathcal{S}$ of $\mathcal{P}$ simply by performing openings. Then, Theorem~\ref{5.3} assures
us that $h(\mathcal{C})=h(\mathcal{S})=0$. Therefore, we get the following translation of Proposition~\ref{56} to the
context of combinatorial collapses.

\begin{figure}
\centering
\includegraphics[scale=0.65]{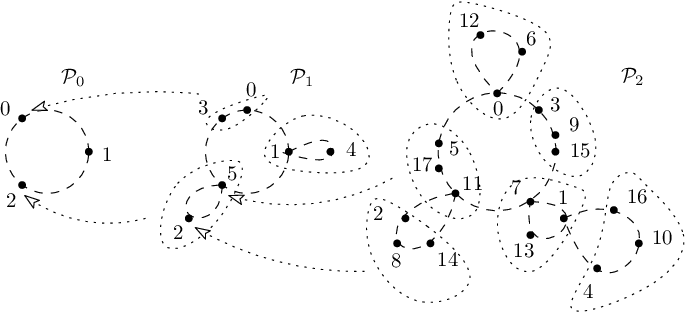}
\caption[fck]{An example of a zero entropy 18-periodic pattern $\mathcal{P}_2$ and the corresponding sequence of collapses.\label{ullh0}}
\end{figure}

\begin{proposition}\label{56-comb}
Let $\mathcal{P}$ be a nontrivial periodic pattern with entropy zero. Then, the combinatorial collapse of $\mathcal{P}$ has entropy zero.
\end{proposition}


\begin{defi}\label{explosions}
As an immediate consequence of Proposition~\ref{56-comb}, a zero entropy $n$-periodic pattern $\mathcal{P}$ has associated
a sequence of patterns $\{\mathcal{P}_i\}_{i=0}^r$ and a sequence of integers $\{p_i\}_{i=0}^r$ for some $r\ge0$ such that:
\begin{enumerate}
\item $\mathcal{P}_r=\mathcal{P}$
\item $\mathcal{P}_0$ is a trivial $p_0$-periodic pattern
\item For $1\le i\le r$, $\mathcal{P}_{i}$ has a maximal separated structure of $\prod_{j=0}^{i-1} p_{j}$ trivial blocks of cardinality $p_i$
and $\mathcal{P}_{i-1}$ is the corresponding combinatorial collapse.
\end{enumerate}
The sequence $\{\mathcal{P}_i\}_{i=0}^r$ will be called \emph{the sequence of collapses of} $\mathcal{P}$. Notice that $\prod_{j=0}^{r}p_j=n$.
See Figure~\ref{ullh0} for an example with $p_0=3$, $p_1=2$, $p_2=3$.
\end{defi}

\begin{remark}\label{persist}
Let $\mathcal{P}$ be a zero entropy $n$-periodic pattern and let $\{\mathcal{P}_i\}_{i=0}^r$ be the corresponding sequence of collapses.
Consider any particular time labeling $\{0,1,\ldots,n-1\}$ of the points of $\mathcal{P}$. By
Remark~\ref{standing3}, this choice fixes the time labels of all points in all patterns of the sequence of collapses. Note also that, for any
$0\le i<r$, the integers labeling the points of $\mathcal{P}_i$ persist as labels of points in $\mathcal{P}_{i+1}$. In particular,
if $p_0$ is the period of the trivial pattern $\mathcal{P}_0$, then $\{0,1,\ldots,p_0-1\}$ are the only integers in the rank $\{0,1,\ldots,n-1\}$
that persist as labels of points in any pattern of the sequence of collapses. See Figure~\ref{ullh0} for an example with $p_0=3$.
\end{remark}

\section{Branching sequences}\label{S6b}
In this Section we dive deeper into the very particular combinatorial structure of zero entropy patterns. The obtained results
will be used in Section~\ref{S7} to prove Theorem~\ref{2inners_3op}.

Let $\mathcal{P}$ be an $n$-periodic pattern and let $x$ be a point of $\mathcal{P}$ of valence $\nu\ge1$. Consider any model $(T,P,f)$ of $\mathcal{P}$. Then, $T\setminus\{x\}$ has $\nu$ connected components $K_1,K_2,\ldots,K_\nu$. We want to register how the forward iterates of the point $x$ are distributed among the connected components of $T\setminus\{x\}$. To this end, consider the integer time labeling of the points of $\mathcal{P}$ such that $x=0$. Now, $\{P\cap K_i\}_{i=1}^\nu$ can be viewed as a partition of $\{1,2,\ldots,n-1\}$. The set $(P\cap K_i)\cup\{0\}$ of points of $\mathcal{P}$ will be called an \emph{$x$-branch}. Note that this notion is independent of the chosen model $(T,P,f)$ representing $\mathcal{P}$. As an example, consider the 7-periodic pattern $\mathcal{P}$ shown in Figure~\ref{brutty2}. Let $x$ be the point of valence 3 labeled as 5 in that
figure. Shift all labels by $-5$ (mod 7). The discrete components of $\mathcal{P}$ read now as $\{0,3,5\}$, $\{3,6\}$, $\{0,2\}$ and $\{0,1,4\}$. The $x$-branches of $\mathcal{P}$ are then $\{0,3,5,6\}$, $\{0,2\}$ and $\{0,1,4\}$.

\begin{remark}\label{isi}
Let $\mathcal{P}$ be a periodic pattern and let $x$ be any point of $\mathcal{P}$. Observe that any discrete component of $\mathcal{P}$ is contained in a single $x$-branch. As a direct consequence of this fact, if in addition $\mathcal{P}$ has entropy zero then any block of the maximal structure of trivial blocks is contained in a single $x$-branch.
\end{remark}

To understand the following result, it is crucial to keep in mind Remarks~\ref{standing3} and \ref{persist} concerning the labeling conventions of points and blocks in zero entropy patterns. In particular, the labels of all points in the combinatorial collapse of a pattern $\mathcal{P}$ persist as labels of points in $\mathcal{P}$.

\begin{lemma}\label{isi2}
Let $\mathcal{P}$ be a zero entropy periodic pattern with a maximal separated structure $P_0\cup P_1\cup\ldots\cup P_{p-1}$ of trivial blocks. Let $\mathcal{C}$ be the combinatorial collapse of $\mathcal{P}$. If $0\le i,j,k<p$ are three points of $\mathcal{C}$ such that $\{j,k\}$ is contained in a single $i$-branch of $\mathcal{C}$, then $P_j\cup P_k$ is contained in a single $i$-branch of $\mathcal{P}$.
\end{lemma}
\begin{proof}
Assume by way of contradiction that $P_j$ and $P_k$ are respectively contained in two different $i$-branches of $\mathcal{P}$. Then, for some $N\ge 2$ there exist $N+1$ different points of $\mathcal{P}$, $x_0,x_1,\ldots,x_N$, such that:
\begin{enumerate}
\item $x_0=j$ and $x_N=k$
\item $x_n$ is inner for all $0<n<N$
\item $\{x_n,x_{n+1}\}$ is contained in a discrete component of $\mathcal{P}$ for $0\le n<N$
\item $x_m=i$ for some $1\le m<N$
\end{enumerate}
Intuitively, the above ordered sequence of points accounts for all points of $\mathcal{P}$ successively met in the shortest path going from $j$ to $k$. The assumption that $i$ separates $j$ from $k$ is imposed by property (d).

Consider now, for any point $x_n$ of the above sequence, the collapse of the trivial block of $\mathcal{P}$ containing $x_n$. It is a point of $\mathcal{C}$ that we denote by $y_n$. Note that $y_0=j$, $y_m=i$ and $y_N=k$. Observe also that, for a pair of consecutive points $x_n,x_{n+1}$, it may happen that $\{x_n,x_{n+1}\}$ is contained in a block. In this case, since the blocks are trivial, $\{x_n,x_{n+1},x_{n+2}\}$ is not contained in a block. Therefore, $y_n=y_{n+1}\ne y_{n+2}$. On the other hand, if $\{x_n,x_{n+1}\}$ is not contained in a block, by the definition of the combinatorial collapse, $\{y_n,y_{n+1}\}$ is a binary set contained in a discrete component of $\mathcal{C}$. This observations lead to the existence of a sequence $z_0,z_1,\ldots,z_{M}$ of $M+1\le N+1$ points of $\mathcal{C}$ such that
\begin{enumerate}[(a')]
\item $z_0=j$ and $z_M=k$
\item $z_n$ is inner for all $0<n<M$
\item $\{z_n,z_{n+1}\}$ is contained in a discrete component of $\mathcal{C}$ for $0\le n<M$
\item $x_{m'}=i$ for some $1\le m'<M$
\end{enumerate}
By property (d'), $j$ and $k$ belong to different $i$-branches in $\mathcal{C}$, in contradiction with the hypothesis of the lemma.
\end{proof}

Let $\mathcal{P}$ be a zero entropy periodic pattern and let $\mathcal{C}$ be its combinatorial collapse. Let us call $\{P_i\}$ and $\{Q_i\}$ the blocks of the respective maximal structures of trivial blocks. Let $x$ be a point of $\mathcal{P}$ and let $P_i$ be the block of $\mathcal{P}$ containing $x$. Let us call $y$ the point of $\mathcal{C}$ corresponding to the collapse of $P_i$ and let $Q_j$ be the block of $\mathcal{C}$ containing $y$. By Remark~\ref{isi}, there exists a unique $x$-branch $Z$ containing $P_i$. On the other hand, Remark~\ref{isi} yields also that $Q_j$ is contained in a single $y$-branch of $\mathcal{C}$. Recall now that the labels of the points in $\mathcal{C}$ persist as labels of points in $\mathcal{P}$. So, we can view $Q_j$ also as a subset of points of $\mathcal{P}$. Then, by Lemma~\ref{isi2}, there exists a unique $x$-branch $Z'$ containing $Q_j$. The point $x$ will be called \emph{bidirectional} if $Z\ne Z'$.

\begin{lemma}\label{critic}
Any periodic pattern with entropy zero has bidirectional inner points.
\end{lemma}
\begin{proof}
Let $\mathcal{P}=([T,P],[f])$ be a zero entropy pattern and let $\mathcal{C}$ be the combinatorial collapse of $\mathcal{P}$. Let
$P_0\cup P_1\cup\ldots\cup P_{p-1}$ and $Q_0\cup Q_1\cup\ldots Q_{q-1}$ be the maximal separated block structures of $\mathcal{P}$ and $\mathcal{C}$ respectively.

Let $x$ be any inner point of $\mathcal{P}$. Assume that $x$ is not bidirectional. In order to do not overload the notation, assume without loss of generality that $x=0$. By the standing labeling conventions, $0\in P_0$ and the collapse of $P_0$ is the point of $\mathcal{C}$ labeled as 0, that belongs to the block $Q_0=\{0,q,2q,\ldots,(p/q-1)q\}$. Since $P_0$ is a trivial block, $P_0\subset C$ for a discrete component $C$ of $\mathcal{P}$. Set
\[ X:=\bigcup_{1\le k<p/q} P_{kq}. \]
The set $X$ is the expansion of all points in $Q_0\setminus\{0\}$ to the corresponding blocks in $\mathcal{P}$. Since we are assuming that 0 is not bidirectional, Remark~\ref{isi} and Lemma~\ref{isi2} imply that
\begin{equation}\label{etiq}
P_0\cup X\mbox{ is contained in a single 0-branch $Z$ of }\mathcal{P}.
\end{equation}

We start by distinguishing two cases.

\begin{case}{1} $X\cap C=\emptyset$. \end{case}
We claim that in this case $C=P_0$. Indeed, $Q_0$ is contained in a discrete component of $\mathcal{C}$. By definition of the combinatorial collapse, all blocks $P_{iq}$ for $0\le i<p/q$ must intersect a single discrete component $D$ of $\mathcal{P}$. Since $X\cap C=\emptyset$, by (\ref{etiq}) this is only possible if $C=P_0$ (as claimed), $D$ is contained in the 0-brach $Z$ and $D$ is adjacent to $C$. See Figure~\ref{figcritic} (center). Let $x'$ be the only point in $C\cap D=P_0\cap D$, whose collapse is the point 0 in $\mathcal{C}$. Then, the $x'$-branch containing $P_0$ and the $x'$-branch containing $Q_0$ are different. Therefore, $x'$ is bidirectional and we are done.

\begin{case}{2} $X\cap C\ne\emptyset\mbox{ and }X\not\subset C$. \end{case}
In this case, all blocks $P_{iq}$ intersect $C$ and at least one block, say $P_{jq}$, has an inner point $x'$ in common with $C$, whose collapse is the point $jq$ in $\mathcal{C}$. See Figure~\ref{figcritic} (right). Then, the $x'$-branch containing $P_{jq}$ and the $x'$-branch containing $Q_0$ are different. Therefore, $x'$ is bidirectional and we are done.

\begin{figure}
\centering
\includegraphics[scale=0.65]{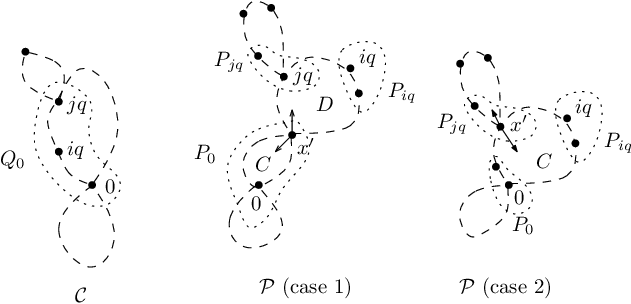}
\caption[fck]{The two cases in the proof of Lemma~\ref{critic}. The arrows mark the two different $x'$-branches implying that $x'$ is bidirectional. \label{figcritic}}
\end{figure}

Note that if $\mathcal{P}$ has no bidirectional inner points, then from above we are not in the hypotheses of cases 1 and 2 and, in consequence, $X\subset C$. Since $P_0\subset C$, we get that
\[ \tilde{P}_0:=P_0\cup X=\bigcup_{0\le k<p/q} P_{kq}\subset D. \]

Set $\tilde{P}_i:=\bigcup_{k=0}^{(p/q)-1} P_{i+kq}$ for $0\le i<q$. From above, if $\mathcal{P}$ has no bidirectional inner points then $\tilde{P}_i$ is contained in a single discrete component of $\mathcal{P}$. Moreover, since $f(\tilde{P}_i)=\tilde{P}_{i+1}$, it follows that $\tilde{P}_0\cup \tilde{P}_1\cup\ldots\cup\tilde{P}_{q-1}$ is a trivial block structure for $\mathcal{P}$, in contradiction with the maximality of the structure $P_0\cup P_1\cup\ldots\cup P_{p-1}$.
\end{proof}

Let $x$ be a point of an $n$-periodic pattern $\mathcal{P}$ and let $\nu\ge1$ be the valence of $x$. It is convenient to fix an indexing of the set of $x$-branches. Next we define a natural indexing method that will be used by default from now on. Recall that, arithmetically, an $x$-branch is nothing but a subset of $\{0,1,\ldots,n-1\}$. Moreover, each $x$-branch contains 0 by definition and the intersection of two different $x$-branches is $\{0\}$. We will index the set of $x$-branches according to the minimum (positive) time distance from $x$ to a point in the branch. More precisely, for any $x$-branch $Z$, let $d_Z$ be the minimum positive integer in $Z$. From now on, we will assume that the set $\{Z_i\}_{i=1}^\nu$ of $x$-branches is indexed in such a way that $d_{Z_i}<d_{Z_j}$ if and only if $i<j$. As an example, consider the 7-periodic pattern $\mathcal{P}$ shown in Figure~\ref{brutty2}. Let $x$ be the point of valence 3 labeled as 5 in that figure. The $x$-branches of $\mathcal{P}$ are then $X=\{0,3,5,6\}$, $Y=\{0,2\}$ and $W=\{0,1,4\}$, with $d_X=3$, $d_Y=2$ and $d_W=1$. So, for this example we would denote the set of $x$-branches as $\{Z_1,Z_2,Z_3\}$, with $Z_1=W$, $Z_2=Y$ and $Z_3=X$.

Let $\mathcal{P}$ be an $n$-periodic pattern and let $x$ be an inner point of $\mathcal{P}$, of valence $\nu>1$. There exists a unique $n$-periodic $\nu$-flower (a pattern with a unique inner point $y$ and $\nu$ discrete components) whose set of $y$-branches, that coincides with its set of discrete components (petals) when $y$ is labeled as 0, coincides with the set of $x$-branches of $\mathcal{P}$. Such a pattern will be denoted by $\mathcal{F}_{x}(\mathcal{P})$. Note that $\mathcal{F}_{x}(\mathcal{P})$ is in some sense the simplest pattern having the set of $x$-branches of $\mathcal{P}$, and is obtained from $\mathcal{P}$ by performing iteratively all possible openings that do not consist of joining two discrete components adjacent at $x$. For an example, consider the 7-periodic pattern $\mathcal{P}$ shown in Figure~\ref{brutty2}. Let $x$ be the point of valence 3 labeled as 5 in that figure. In this case, $\mathcal{F}_{x}(\mathcal{P})$ is the 3-flower whose petals are $\{5,1,3,4\}$, $\{5,0\}$ and $\{5,2,6\}$. After shifting the labels by $-5$ (mod 7) in order that the central point of the flower reads as 0, the petals are written as $\{0,3,5,6\}$, $\{0,2\}$ and $\{0,4,1\}$, that are precisely the $x$-branches of $\mathcal{P}$.

\begin{remark}\label{trivialet}
Let $\mathcal{P},\mathcal{Q}$ be $n$-periodic patterns. For any $x,y\in\{0,1,\ldots,n-1\}$, the set of $x$-branches of $\mathcal{P}$ and the set of $y$-branches of $\mathcal{Q}$ coincide if and only if $\mathcal{F}_{x}(\mathcal{P})=\mathcal{F}_{y}(\mathcal{Q})$.
\end{remark}

The previous remark says that in fact the notation $\mathcal{F}_{x}(\mathcal{P})$, that denotes a pattern, could have been reserved to denote simply the (arithmetic) set of $x$-branches of $\mathcal{P}$. We have used the construction of the flower just as a trick that hopefully supports the geometric visualization.

The following result is true for any point of a periodic pattern but, in pursuit of simplicity, is stated without loss of generality for a point labeled as 0.

\begin{lemma}\label{isi3}
Let $\mathcal{P}$ be a zero entropy periodic pattern and let $\{\mathcal{P}_i\}_{i=0}^r$ be the associated sequence of collapses. For any $0\le i\le r$, let $P_0^i$ be the block of the maximal structure of $\mathcal{P}_i$ containing the point $0\in\mathcal{P}_i$. Then, $P_0^i$ is contained in a single 0-branch of $\mathcal{P}$.
\end{lemma}
\begin{proof}
By Proposition~\ref{56-comb}, $h(\mathcal{P}_i)=0$. Then, by Remark~\ref{isi}, $P_0^i$ is contained in a single 0-branch of $\mathcal{P}_i$. The result follows then immediately by using iteratively Lemma~\ref{isi2}.
\end{proof}

A sequence $\{(p_i,\delta_i)\}_{i=0}^r$ of pairs of integers will be called a \emph{branching sequence} if the following conditions hold:
\begin{enumerate}
\item[(bs1)] $p_i\ge 2$ for $0\le i\le r$.
\item[(bs2)] $\delta_1=1$.
\item[(bs3)] For any $1\le i\le r$, if $\delta_i\notin\{\delta_j\}_{j=0}^{i-1}$ then $\delta_i=1+\max\{\delta_j\}_{j=0}^{i-1}$.
\end{enumerate}

Let $\mathcal{P}$ be a zero entropy $n$-periodic pattern and let $\{\mathcal{P}_i\}_{i=0}^r$ be the associated sequence of collapses. Let $x$ be any point of $\mathcal{P}$, with valence $\nu\ge1$. Relabel the points of $\mathcal{P}$ in such a way that $x=0$. Now, for any pattern $\mathcal{P}_i$ in the sequence of collapses, Lemma~\ref{isi3} tells us that the block of the maximal structure of $\mathcal{P}_i$ containing $0$ is contained in a single 0-branch $\delta_i$ of $\mathcal{P}$. It is easy to check that the sequence $\{(p_i,\delta_i)\}_{i=0}^r$, where $p_0$ is the period of $\mathcal{P}_0$ and $p_i$ is the cardinality of the blocks of the maximal structure in $\mathcal{P}_i$ for any $1\le i\le r$, satisfies properties (bs1--3) above. It will be called \emph{the branching sequence of $\mathcal{P}$ around $x$}. Once the indexing of the $x$-branches is fixed after the accorded convention, it is uniquely determined by the pattern $\mathcal{P}$ and the chosen point $x$ of $\mathcal{P}$. See Figure~\ref{bsfig} for an example of construction of the branching sequence. For the pattern $\mathcal{P}$ shown in that figure, the 0-branches are $Z_1=\{0,1,2,3,5,6,7,8,9,10,11,13,14,15\}$ and $Z_2=\{0,4,12\}$. The maximal trivial blocks have cardinality 2 in each pattern of the sequence of collapses. The blocks containing 0 are $\{0,1\}$ in $\mathcal{P}_0$, $\{0,2\}$ in $\mathcal{P}_1$, $\{0,4\}$ in $\mathcal{P}_2$ and $\{0,8\}$ in $\mathcal{P}$. Seen as sets of points of $\mathcal{P}$, they are respectively contained in $Z_1$, $Z_1$, $Z_2$ and $Z_1$. Collecting it all, we get that the branching sequence of $\mathcal{P}$ around 0 is $\{(2,1),(2,1),(2,2),(2,1)\}$.

\begin{figure}
\centering
\includegraphics[scale=0.65]{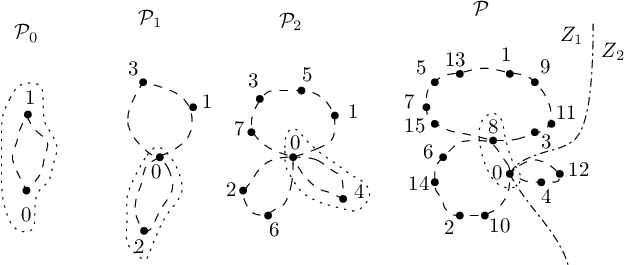}
\caption[fck]{A pattern $\mathcal{P}$ whose branching sequence around 0 is $\{(2,1),(2,1),(2,2),(2,1)\}$. The two 0-branches in $\mathcal{P}$ are denoted with $Z_1$ and $Z_2$ with the standard indexing convention.\label{bsfig}}
\end{figure}

The following observation follows directly from the definitions.

\begin{remark}\label{bidi}
Let $\{(p_i,\delta_i)\}_{i=0}^r$ be the branching sequence of a zero entropy pattern around an inner point $x$. Then, $x$ is bidirectional if and only if $\delta_{r-1}\ne\delta_r$.
\end{remark}

%
%
Now we reverse the process and consider an (abstract) branching sequence $S=\{(p_i,\delta_i)\}_{i=0}^r$. Let us see that from such a sequence we can construct a zero entropy $n$-periodic $\nu$-flower, where $n=p_0p_1\cdots p_r$ and $\nu=\max\{\delta_i\}_{i=0}^r$. Consider a $p_0$-periodic trivial pattern $\mathcal{P}_0$ and let us denote its unique discrete component by $C^0_{\delta_1}=C^0_1$ (property (bs2)). Assume now that a zero entropy periodic pattern $\mathcal{P}_i$ of period $p_0p_1\cdots p_i$ has been defined, with $d_i:=\max\{\delta_j\}_{j=0}^i$ discrete components labeled as $\{C^i_1,C^i_2,\ldots,C^i_{d_i}\}$, all adjacent to the point 0. Now we define a new pattern $\mathcal{P}_{i+1}$ of period $p_0p_1\cdots p_{i+1}$ by applying the following procedure. For any point $j$ of $\mathcal{P}_i$, set $K_j:=\{j+p_i,j+2p_i,\ldots,j+(p_{i+1}-1)p_i\}$. Note that, by (bs3), either $\delta_{i+1}\le d_i$, and in this case we set $d_{i+1}:=d_i$, or $\delta_{i+1}=d_i+1$, and in this case we set $d_{i+1}:=d_i+1$. The pattern $\mathcal{P}_{i+1}$ is then defined as a $d_{i+1}$-flower with inner point 0 and discrete components labeled as $\{C^{i+1}_1,C^{i+1}_2,\ldots,C^{i+1}_{d_i+1}\}$, in such a way that $K_0\subset C^{i+1}_{\delta_{i+1}}$ and for any point $j\ne 0$ of $\mathcal{P}_i$, $K_j\subset C^{i+1}_k$ if and only if $j\in C^i_k$. By iterating $r$ times this procedure, finally we obtain the prescribed $\nu$-flower $\mathcal{P}_r$, with the inner point conventionally labeled as 0 by construction. Such a flower, algorithmically constructed from the branching sequence $S$, will be denoted by $\mathcal{F}(S)$. To fit the intuition into the description of the algorithm, note that the combinatorial collapse of a zero entropy $k$-flower is either a $(k-1)$-flower when a petal fully coincides with a block of the maximal structure, and a $k$-flower otherwise.

\begin{example}\label{crucy}
Let $S=\{(2,1),(3,2),(2,2),(2,3)\}$. In Figure~\ref{bsalgorithm} we have shown the sequence of patterns leading to $\mathcal{F}(S)$ according to the prescribed algorithm.
\end{example}

\begin{figure}
\centering
\includegraphics[scale=0.65]{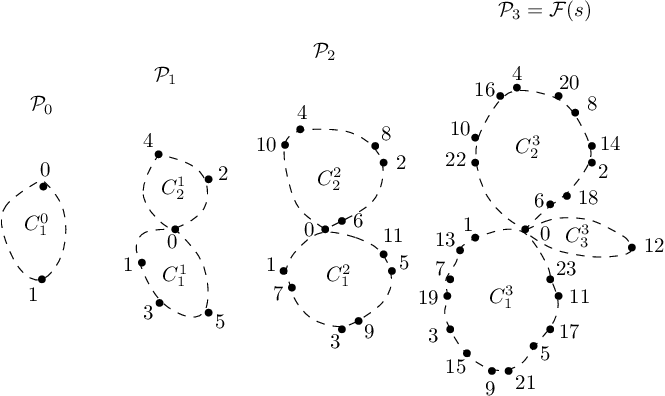}
\caption[fck]{The steps of the algorithm to generate the flower $\mathcal{F}(S)$ from the branching sequence $S=\{(2,1),(3,2),(2,2),(2,3)\}$.\label{bsalgorithm}}
\end{figure}


A branching sequence $S=\{(p_i,\delta_i)\}_{i=0}^r$ will be called \emph{minimal} if $\delta_{i+1}\ne\delta_i$ for all $0\le i<r$.

\begin{lemma}\label{bra0}
Let $S$ and $R$ be minimal branching sequences such that $\mathcal{F}(S)=\mathcal{F}(R)$. Then $S=R$, i.e. $S$ and $R$ have the same length and are identical term by term.
\end{lemma}
\begin{proof}
Set $S=\{(p_i,\delta_i)\}_{i=0}^r$ and $R=\{(q_i,\kappa_i)\}_{i=0}^{t}$. By Remark~\ref{trivialet}, the hypothesis that $\mathcal{F}(S)$ and $\mathcal{F}(R)$ are the same pattern can be reworded as follows: if both flowers are labeled in such a way that the respective inner points read as 0, then the respective sets of 0-branches coincide. In particular,
\begin{equation}\label{producte}
\prod_{i=0}^r p_i=\prod_{i=0}^t q_i.
\end{equation}

First we claim that $(p_1,\delta_1)=(q_1,\kappa_1)$. Indeed, by property (bs2), $\delta_1=\kappa_1=1$. Assume by way of contradiction that $p_1<q_1$ (the argument is symmetric when $q_1<p_1$). Then, from (\ref{producte}) it follows that $r\ge2$. Moreover, since $S$ is minimal,  $\delta_2\ne\delta_1$. Property (bs3) yields then that $\delta_2=2$. So, the algorithm of construction of $\mathcal{F}(S)$ and $\mathcal{F}(R)$ implies that the 0-branch indexed as 1 in $\mathcal{F}(S)$ contains the points $0,1,2,\ldots,p_1-1$ and the point $p_1$ is contained in the 0-branch  indexed as 2, while the 0-branch indexed as 1 in $\mathcal{F}(R)$ contains at least the points $0,1,2,\ldots,p_1-1,p_1$. In consequence, $\mathcal{F}(S)$ and $\mathcal{F}(R)$ are not the same pattern, a contradiction that proves the claim.

Assume now that all terms of $S$ and $R$ are identical up to an index $j\ge 1$ (the previous claim states that this is true when $j=1$). In this case, if $S$ has length $j$, then (\ref{producte}) implies that $R$ has also length $j$ and we are done. Assume that $r>j$ (the arguments and conclusions are the same if $t>j$). Set $k:=\prod_{i=0}^j p_i=\prod_{i=0}^j q_i$. From the algorithm of construction of $\mathcal{F}(S)$ and $\mathcal{F}(R)$, it follows that all points from 0 to $k-1$ are distributed identically inside the 0-branches of both flowers. The same arguments used above show then that $t>j$, and that if we assume $(p_{j+1},\delta_{j+1})\ne(q_{j+1},\kappa_{j+1})$, we reach a contradiction since the points $k,k+1,k+2,\ldots,kp_{j+1}-1$ will be distributed in different 0-branches of $\mathcal{F}(S)$ and $\mathcal{F}(R)$.
\end{proof}

\begin{remark}\label{florsminimals}
If $\mathcal{P}$ is a zero entropy flower, then the branching sequence of $\mathcal{P}$ around its unique inner point is minimal. Indeed, if for an index $i$ we had two consecutive terms $(p_i,\delta_i)$, $(p_{i+1},\delta_{i+1})$ with $\delta_i=\delta_{i+1}$, then, in the sequence $\{\mathcal{P}_i\}_{i=0}^r$ of collapses, the trivial blocks for the pattern $\mathcal{P}_i$ would not be maximal, since there would exist greater trivial blocks of cardinality $p_ip_{i+1}$. For example, let $\mathcal{P}$ be the rightmost pattern shown in Figure~\ref{bsalgorithm}, that is in fact the 3-flower constructed from $S=\{(2,1),(3,2),(2,2),(2,3)\}$. The sequence of collapses of $\mathcal{P}$ is \emph{not} $\{\mathcal{P}_i\}_{i=0}^3$ but $\{\mathcal{P}'_i\}_{i=0}^2$, with $\mathcal{P}'_0=\mathcal{P}_0$, $\mathcal{P}'_1=\mathcal{P}_2$ and $\mathcal{P}'_2=\mathcal{P}_3$. The branching sequence of $\mathcal{P}$ around 0 is then $S'=\{(2,1),(6,2),(2,3)\}$, which is minimal.
\end{remark}

Let $S=\{(p_i,\delta_i)\}_{i=0}^r$ be a branching sequence. Assume that $S$ is not minimal, i.e. for some $0\le j<r$ we have that $\delta_{j+1}=\delta_j$. Then we can consider a \emph{reduced sequence} $S'=\{(p'_i,\delta'_i)\}_{i=0}^{r-1}$ defined as $(p'_i,\delta'_i)=(p_i,\delta_i)$ for $0\le i<j$, $(p'_j,\delta'_j)=(p_jp_{j+1},\delta_j)$ and $(p'_i,\delta'_i)=(p_{i+1},\delta_{i+1})$ for $j<i\le r-1$. One can easily check that $S'$ satisfies (bs1--3) and is thus a branching sequence. The following result states that $S$ and $S'$ generate the same flower. It follows immediately from the algorithm of construction of $\mathcal{F}(S)$.

\begin{lemma}\label{bra1}
Let $S,S'$ be branching sequences such that $S'$ has been reduced from $S$. Then, $\mathcal{F}(S')=\mathcal{F}(S)$.
\end{lemma}

The process of reducing a non-minimal branching sequence $S=\{(p_i,\delta_i)\}_{i=0}^r$ can be iterated as many times as necessary in order to finally obtain what we call the \emph{sequence fully reduced from $S$}, a minimal branching sequence  $\widehat{S}=\{(\widehat{p}_i,\widehat{\delta}_i)\}_{i=0}^{\widehat{r}}$ satisfying $\prod_{i=0}^r p_i=\prod_{i=0}^{\widehat{r}} \widehat{p}_i$.
One can easily check that it is unique and well defined. As a direct corollary of Lemma~\ref{bra1}, we get the following result.

\begin{corollary}\label{bra2}
Let $S$ be a branching sequence and let $\widehat{S}$ be the sequence fully reduced from $S$. Then, $\mathcal{F}(S)=\mathcal{F}(\widehat{S})$.
\end{corollary}

In this section we have defined two procedures to generate a flower (equivalently, a set of branches). The first one uses openings to get a flower
$\mathcal{F}_{x}(\mathcal{P})$ given a pattern $\mathcal{P}$ and a point $x$ of $\mathcal{P}$, while the second one constructs a flower $\mathcal{F}(S)$ given an abstract branching sequence $S$. The next lemma, that follows immediately from the definitions and the labeling conventions of the points and branches, states that if $S$ is precisely the branching sequence of $\mathcal{P}$ around $x$, both flowers are the same as patterns.

\begin{lemma}\label{bra3}
Let $\mathcal{P}$ be a zero entropy pattern. Let $x$ be a point of $\mathcal{P}$ and let $S$ be the branching sequence of $\mathcal{P}$ around $x$. Then, $\mathcal{F}(S)=\mathcal{F}_{x}(\mathcal{P})$.
\end{lemma}

Now we are ready to use all techniques and results of this section to get the following proposition and the subsequent corollary, that will be crucial in the proof of Theorem~\ref{2inners_3op}.

\begin{proposition}\label{newresult}
Let $\mathcal{P}$ be a zero entropy periodic pattern and let $x$ be a point of $\mathcal{P}$. Let $S$ be the branching sequence of $\mathcal{P}$ around $x$ and let $\widehat{S}$ be the sequence fully reduced from $S$. Then, the branching sequence of $\mathcal{F}_x(\mathcal{P})$ around $x$ is $\widehat{S}$.
\end{proposition}
\begin{proof}
By Lemma~\ref{bra3},
\begin{equation}\label{part1}
\mathcal{F}(S)=\mathcal{F}_x(\mathcal{P}).
\end{equation}
Let $R$ be the branching sequence of $\mathcal{F}_x(\mathcal{P})$ around $x$. We want to see that $R=\widehat{S}$. Since $\mathcal{F}_x(\mathcal{F}_x(\mathcal{P}))=\mathcal{F}_x(\mathcal{P})$, using again Lemma~\ref{bra3} yields
\begin{equation}\label{part2}
\mathcal{F}(\mathcal{R})=\mathcal{F}_x(\mathcal{P}).
\end{equation}
On the other hand, by Corollary~\ref{bra2},
\begin{equation}\label{part3}
\mathcal{F}(S)=\mathcal{F}(\widehat{S}).
\end{equation}
From (\ref{part1}), (\ref{part2}) and (\ref{part3}) we get then that
\begin{equation}\label{part4}
\mathcal{F}(R)=\mathcal{R}(\widehat{S}).
\end{equation}
Since $\widehat{S}$ is minimal by definition of a fully reduced sequence and $R$ is minimal by Remark~\ref{florsminimals}, then (\ref{part4}) and Lemma~\ref{bra0} imply that $R=\widehat{S}$.
\end{proof}

\begin{corollary}\label{bra4}
Let $\mathcal{P}$ and $\mathcal{Q}$ be two zero entropy $n$-periodic patterns. Let $x$ and $y$ be inner points of $\mathcal{P}$ and $\mathcal{Q}$ respectively. Let $\widehat{S}$ and $\widehat{R}$ be the fully reduced sequences of $\mathcal{P}$ and $\mathcal{Q}$ around $x$ and $y$ respectively. If $\mathcal{F}_{x}(\mathcal{P})=\mathcal{F}_{y}(\mathcal{Q})$ then $\widehat{S}=\widehat{R}$, i.e. both sequences have the same length and are identical term by term.
\end{corollary}

\section{Proof of Theorem~\ref{2inners_3op}}\label{S7}
Recall that the hypothesis of Theorem~\ref{2inners_3op} is that we have an $n$-periodic pattern $\mathcal{P}$ with at least two inner points and at least three openings. Moreover, $h(\mathcal{P})>0$ and any opening has entropy zero. Under these conditions, we have to prove that $\mathcal{P}$ is $\pi$-reducible for some basic path $\pi$. The name $\mathcal{O}$ that we will use to denote zero entropy patterns in this section stands for \emph{opening}, in the spirit of Theorem~\ref{2inners_3op}.

The following is a simple remark about how the set of $x$-branches, where $x$ is a point of a pattern $\mathcal{P}$, can change after performing an opening of $\mathcal{P}$.

\begin{remark}\label{branchop}
Let $x$ be a point of a pattern $\mathcal{P}$ and let $\mathcal{O}$ be an opening of $\mathcal{P}$. If $\mathcal{O}$ has been obtained by joining two discrete components not adjacent at $x$
(equivalently, the valence of $x$ in $\mathcal{O}$ equals the valence of $x$ in $\mathcal{P}$), then $\mathcal{F}_x(\mathcal{O})=\mathcal{F}_x(\mathcal{P})$. As an example, consider the pattern $\mathcal{P}$ and the opening $\mathcal{O}$ shown in Figure~\ref{brutty2}. Take $x=1$. In this case, $\mathcal{F}_x(\mathcal{O})=\mathcal{F}_x(\mathcal{P})$ is a 2-flower whose petals can be labeled as $\{0,3\}$ and $\{0,1,2,4,5,6\}$. On the other hand, if $\mathcal{O}$ has been obtained by joining two discrete components adjacent at $x$ (equivalently, the valence of $x$ in $\mathcal{O}$ is one less than the valence of $x$ in $\mathcal{P}$), then $\mathcal{F}_x(\mathcal{O})$ is an opening of $\mathcal{F}_x(\mathcal{P})$. As an example, take $x=5$ in the previous example. Here $\mathcal{F}_x(\mathcal{P})$ is a 3-flower whose petals can be labeled as $\{0,3,5,6\}$, $\{0,2\}$ and $\{0,1,4\}$, while $\mathcal{F}_x(\mathcal{O})$ is a 2-flower whose petals can be labeled as $\{0,3,5,6\}$ and $\{0,1,2,4\}$, i.e. an opening of $\mathcal{F}_x(\mathcal{P})$.
\end{remark}

Recall that the integer labels of the points of a pattern $\mathcal{P}$ are by default preserved when performing an opening of $\mathcal{P}$. So, in the following statement we use the same letter $x$ to refer indistinctly to a point of a pattern and to the corresponding point of an opening.

\begin{lemma}\label{2inners_3op_QR}
Let $\mathcal{P}$ be an $n$-periodic pattern with positive entropy such that any opening of $\mathcal{P}$ has entropy zero.
Assume that $\mathcal{P}$ has at least two inner points and at least three openings. Then, there exist a point $x$ of $\mathcal{P}$ and
two different openings $\mathcal{O}$ and $\mathcal{R}$ of $\mathcal{P}$ such that:
\begin{enumerate}
\item $x$ is a bidirectional inner point in $\mathcal{O}$.
\item $x$ is an inner point in $\mathcal{R}$.
\item One of the following statements holds:
\begin{enumerate}
\item[(c1)] $\mathcal{F}_x(\mathcal{O})=\mathcal{F}_x(\mathcal{R})$
\item[(c2)] $\mathcal{F}_x(\mathcal{O})$ is an opening of $\mathcal{F}_x(\mathcal{R})$.
\end{enumerate}
\end{enumerate}
\end{lemma}
\begin{proof}
To prove the result we consider two cases.

\begin{case}{1} $\mathcal{P}$ has exactly two inner points. \end{case}
In this case, the hypothesis imply that at least one inner point has valence larger than 2 and that $\mathcal{P}$ has at least four different openings. Let us consider for instance that $\mathcal{P}$ has one inner $\alpha$ with valence 2 and one inner $\beta$ with valence 3. The proof can be trivially extended to any other case. In this situation, $\mathcal{P}$ has four discrete
components, which we label by $C_0$, $C_1$, $C_2$ and $C_3$. See Figure~\ref{fig2} for a representation of $\mathcal{P}$ and the three openings that we will use below.

According to the notation in Figure~\ref{fig2} we consider $\mathcal{O}$ to be the opening of $\mathcal{P}$ corresponding to the union $C_0\cup C_2$. The pattern $\mathcal{O}$ is a triple chain with two inner points $\alpha$ and $\beta$. Let $x$ be a bidirectional inner point of $\mathcal{O}$, that exists by Proposition~\ref{critic}. Then, (a) holds. Consider now a relabeling of the points of $\mathcal{P}$ (and, in consequence, of $\mathcal{O}$) such that $x=0$. We have now two possibilities.

\begin{figure}
\centering
\includegraphics[scale=0.65]{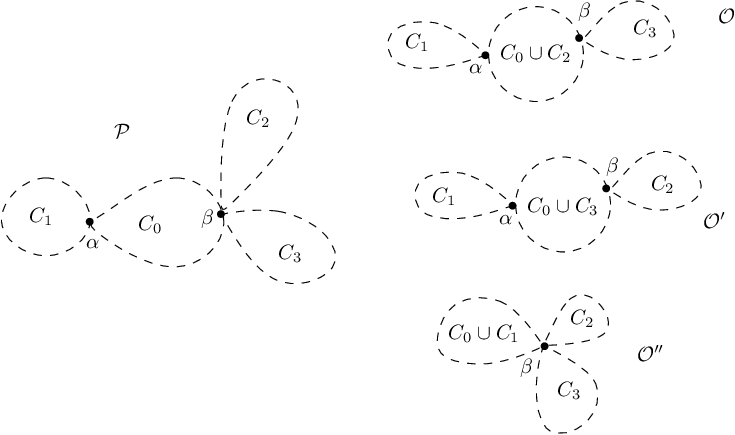}
\caption[fck]{Three possible openings for case 1 in the proof of Lemma~\ref{2inners_3op_QR}.\label{fig2}}
\end{figure}

If $0=\alpha$ then we take $\mathcal{R}$ as the opening $\mathcal{O}'$ corresponding to the union $C_0\cup C_3$. So, (b) is satisfied. Moreover, neither $\mathcal{O}$ nor $\mathcal{R}$ have been formed by joining together discrete components adjacent to 0. It follows that the valence of 0 in $\mathcal{P}$, $\mathcal{O}$ and $\mathcal{R}$ is the same and (c1) follows from Remark~\ref{branchop}.

If $0=\beta$ then we take $\mathcal{R}$ as the opening $\mathcal{O}''$ corresponding to the union $C_0\cup C_1$. So, (b) is satisfied. Moreover, $\mathcal{R}$ has only one inner point, $\beta=0$. In this case the valence of $0$ in $\mathcal{R}$ equals the valence of $0$ in $\mathcal{P}$ and it is one larger than in $\mathcal{O}$. Thus, (c2) follows from Remark~\ref{branchop}.

\begin{case}{2} $\mathcal{P}$ has at least three inner points. \end{case}

Let $\mathcal{O}$ be an arbitrary opening of $\mathcal{P}$. Let $x$ be a bidirectional inner point of $\mathcal{O}$, that exists by Proposition~\ref{critic}. Then, (a) holds. Consider now a relabeling of the points of $\mathcal{P}$ (and, in consequence, of $\mathcal{O}$) such that $x=0$. Let $\alpha\ne0$ and $\beta\ne0$ be two different inner points of $\mathcal{P}$.

If $\mathcal{O}$ has been obtained by joining two discrete components adjacent to $\alpha$, then we choose $\mathcal{R}$ as any opening obtained by joining two discrete components adjacent to $\beta$. In this case, the valence of $0$ is the same in the three patterns $\mathcal{P}$, $\mathcal{O}$ and $\mathcal{R}$ (See Figure~\ref{fig3}). Thus, (b) holds and, by Remark~\ref{branchop}, (c1) is also satisfied.

\begin{figure}
\centering
\includegraphics[scale=0.65]{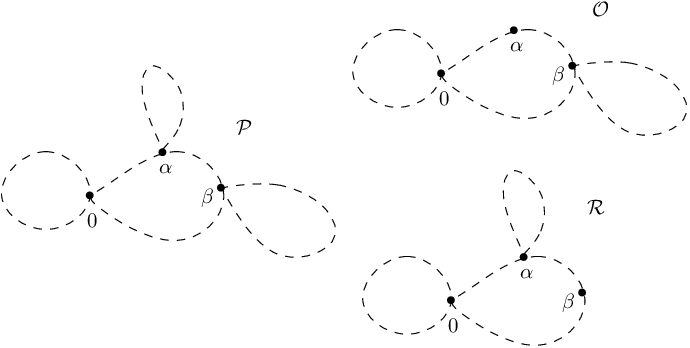}
\caption[fck]{Illustration of case 2 (first subcase) in the proof of Lemma~\ref{2inners_3op_QR}.\label{fig3}}
\end{figure}

\begin{figure}
\centering
\includegraphics[scale=0.65]{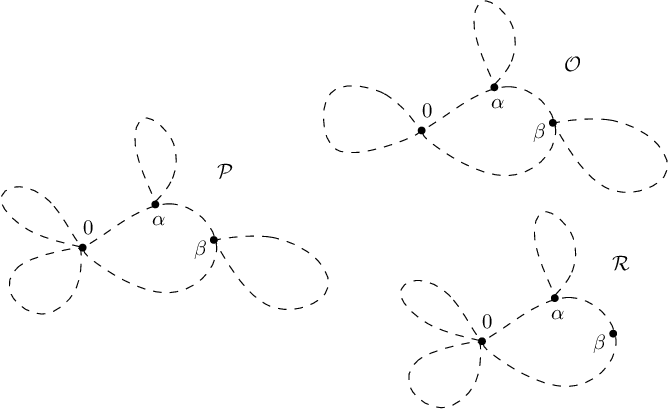}
\caption[fck]{Illustration of case 2 (second subcase) in the proof of Lemma~\ref{2inners_3op_QR}.\label{fig4}}
\end{figure}

Finally, if $\mathcal{O}$ has been formed by joining two discrete components adjacent to $0$, then we choose $\mathcal{R}$ as an opening obtained by joining two discrete components adjacent to $\beta$. In this case, the valence of $0$ in $\mathcal{P}$ and $\mathcal{R}$ is the same and one larger than the valence of $0$ in $\mathcal{O}$. In particular, the valence of $0$ in $\mathcal{P}$ and $\mathcal{R}$ is larger than two. Thus, (b) holds and, by Remark~\ref{branchop}, (c2) is satisfied (see Figure~\ref{fig4}).
\end{proof}

To prove Theorem~\ref{2inners_3op}, we will use branching sequences in the two situations (c1) and (c2) given by Lemma~\ref{2inners_3op_QR}(c). To deal with (c2), we need to relate the branching sequences of both a flower $\mathcal{F}$ and an opening $\mathcal{F}'$ of $\mathcal{F}$.

\begin{lemma}\label{2flors}
Let $\mathcal{F}$ be a zero entropy periodic $\nu$-flower and let $R=\{(q_i,\kappa_i)\}_{i=0}^t$ be the branching sequence of $\mathcal{F}$ around its unique inner point $x$. Let $\mathcal{F}'$ be an opening of $\mathcal{F}$ obtained by joining two discrete components corresponding to two $x$-branches labeled as $j_1,j_2$, with $1\le j_1<j_2\le\nu$. Set $R':=\{(q_i,\kappa'_i)\}_{i=0}^t$, with $\kappa'_i$ defined as

\[ \kappa'_i=\left\{ \begin{array}{lcl}
\kappa_i & \mbox{if} & \kappa_i<j_2 \\
j_1 & \mbox{if} & \kappa_i=j_2 \\
\kappa_i-1 & \mbox{if} & \kappa_i>j_2
\end{array}\right. \]
Then, $R'$ is a branching sequence and the sequence fully reduced from $R'$ is the branching sequence of $\mathcal{F}'$ around $x$.
\end{lemma}
\begin{proof}
It is easy to check directly from the definition of $R'$ that properties (bs1--3) satisfied by $R$ are inherited by $R'$. Thus, $R'$ is a branching sequence. By checking the steps of the algorithm of construction of the flower $\mathcal{F}(R')$, one easily gets that $\mathcal{F}(R')=\mathcal{F}'$.

Let $\widehat{R'}$ be the sequence fully reduced from $R'$. By Corollary~\ref{bra2}, $\mathcal{F}'=\mathcal{F}(R')=\mathcal{F}(\widehat{R'})$. Let $B$ be the branching sequence of $\mathcal{F}'$ around its unique inner point $x$. We want to see that $B=\widehat{R'}$. Since $\mathcal{F}_x(\mathcal{F}')=\mathcal{F}'$, Lemma~\ref{bra3} yields $\mathcal{F}(B)=\mathcal{F}'$. Therefore, $\mathcal{F}(B)=\mathcal{F}(\widehat{R'})$. Since $\widehat{R'}$ is minimal by definition of a fully reduced sequence and $B$ is minimal by Remark~\ref{florsminimals}, the previous equality and Lemma~\ref{bra0} imply $B=\widehat{R'}$.
\end{proof}

To illustrate Lemma~\ref{2flors}, let $\mathcal{F}$ be the 4-flower shown in Figure~\ref{fig2flors}. The discrete components (equivalently, the 0-branches) of $\mathcal{F}$ are $Z_1=\{0,1,3,5,7,9,11,13,15\}$, $Z_2=\{0,2,6,10,14\}$, $Z_3=\{0,4,12\}$, $Z_4=\{0,8\}$. One can check that the branching sequence of $\mathcal{F}$ around 0 is $R=\{(2,1),(2,2),(2,3),(2,4)\}$. Now let $\mathcal{F}'$ be the opening obtained by joining the discrete components $Z_1$ and $Z_3$. The 0-branches of $\mathcal{F}'$, indexed according to the standing convention, are then $Y_1=Z_1\cup Z_3$, $Y_2=Z_2$, $Y_3=Z_4$. The sequence $R'$ defined in the statement of Lemma~\ref{2flors} is $R'=\{(2,1),(2,2),(2,1),(2,3)\}$, that is minimal. According to Lemma~\ref{2flors}, it is the branching sequence of $\mathcal{F}'$ around 0. As another example, let $\mathcal{F}'$ be the opening of $\mathcal{F}$ obtained by joining the discrete components $Z_2$ and $Z_3$. In this case, the 0-branches of $\mathcal{F}'$ are $Y_1=Z_1$, $Y_2=Z_2\cup Z_3$ and $Y_3=Z_4$. The sequence $R'$ defined in the statement of Lemma~\ref{2flors} reads as $R'=\{(2,1),(2,2),(2,2),(2,3)\}$ and its fully reduced sequence $\{(2,1),(4,2),(2,3)\}$ is the branching sequence of $\mathcal{F}'$ around 0.

\begin{figure}
\centering
\includegraphics[scale=0.65]{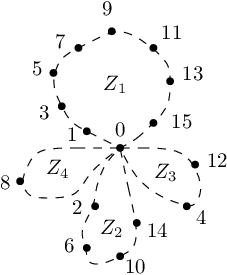}
\caption[fck]{A 16-periodic 4-flower with entropy zero.\label{fig2flors}}
\end{figure}

Now we are in position of proving Theorem~\ref{2inners_3op}.

\begin{proof}[Proof of Theorem~\ref{2inners_3op}]
Let $\mathcal{O}$ and $\mathcal{R}$ be the two openings of $\mathcal{P}$ given by Lemma~\ref{2inners_3op_QR}, let $S=\{(p_i,\delta_i)\}_{i=0}^r$ and $R=\{(q_i,\kappa_i)\}_{i=0}^t$ be the corresponding branching sequences around $x$, and let $\widehat{S}=\{(\widehat{p}_i,\widehat{\delta}_i)\}_{i=0}^{\widehat{r}}$ and $\widehat{R}=\{(\widehat{q}_i,\widehat{\kappa}_i)\}_{i=0}^{\widehat{t}}$ be the sequences fully reduced, respectively, from $S$ and $R$.
From the definition of a reduced sequence,
\begin{equation}\label{equa1}
\widehat{q}_{\widehat{t}}=q_{t-j}q_{t-j+1}\cdots q_{t-1}q_{t}\mbox{ for some }j\ge0.
\end{equation}
On the other hand, since $x$ is bidirectional in $\mathcal{O}$, then, by Remark~\ref{bidi}, $\delta_{r-1}\neq\delta_r$. Therefore, using again the definition of a reduced sequence we get
\begin{equation}\label{equa2}
(\widehat{p}_{\widehat{r}},\widehat{\delta}_{\widehat{r}})=(p_r,\delta_r).
\end{equation}

We claim that $q_t$ divides $p_r$. To prove this claim we will consider the two cases produced by Lemma~\ref{2inners_3op_QR}(c).

Assume first that Lemma~\ref{2inners_3op_QR}(c1) holds. Then, by Corollary~\ref{bra4}, $\widehat{S}$ and $\widehat{R}$ are identical term by term. In particular, $\widehat{q}_{\widehat{t}}=\widehat{p}_{\widehat{r}}$, which is equal to $p_r$ by (\ref{equa2}). Thus, (\ref{equa1}) implies that $q_t$ divides $p_r$, as claimed.

Assume now that Lemma~\ref{2inners_3op_QR}(c2) holds. From Proposition~\ref{newresult} we have that $\widehat{S}$ is the branching sequence of the flower $\mathcal{F}':=\mathcal{F}_x(\mathcal{O})$ and $\widehat{R}$ is the branching sequence of the flower $\mathcal{F}:=\mathcal{F}_x(\mathcal{R})$. Since $\mathcal{F}'$ is an opening of $\mathcal{F}$, Lemma~\ref{2flors} tells us that $\widehat{S}=\{(\widehat{p}_i,\widehat{\delta}_i)\}_{i=0}^{\widehat{r}}$ has been obtained from $\widehat{R}=\{(\widehat{q}_i,\widehat{\kappa}_i)\}_{i=0}^{\widehat{t}}$ in two steps. First, we consider a sequence $\widehat{R}'=\{(\widehat{q}_i,\widehat{\kappa}'_i)\}_{i=0}^{\widehat{t}}$ and then fully reduce it to obtain $\{(\widehat{p}_i,\widehat{\delta}_i)\}_{i=0}^{\widehat{r}}$. Again the definition of a reduction implies that
\[ \widehat{p}_{\widehat{r}}=\widehat{q}_{\widehat{t}-\ell}\widehat{q}_{\widehat{t}-\ell+1}\cdots\widehat{q}_{\widehat{t}-1}\widehat{q}_{\widehat{t}}\mbox{ for some }\ell\ge0. \]
The previous equality and (\ref{equa1}) imply that $q_t$ divides $p_r$ also in this case. In consequence, the claim is proved.

To end up we claim that the divisibility of $p_r$ by $q_t$ implies the $\pi$-reducibility of $\mathcal{P}$. We recall that $p_r$ and $q_t$ are the cardinalities of the trivial blocks in the respective maximal structures of $\mathcal{O}$ and $\mathcal{R}$ given by Proposition~\ref{56-comb}. Relabel if necessary the points of $\mathcal{P}$ in such a way that $x=0$. The inner point $0$ belongs to the block of $\mathcal{O}$
\[
O_0=\{0,\tfrac{n}{p_r},\tfrac{2n}{p_r},\dots,\tfrac{(p_r-1)n}{p_r}\}.
\]
By Proposition~\ref{twocharact}, $\mathcal{O}$ is $\pi$-reducible for any basic path $\pi$ contained in $O_0$. On the other hand, the inner point $0$ belongs to the block of $\mathcal{R}$
\[
R_0=\{0,\tfrac{n}{q_t},\tfrac{2n}{q_t},\dots,\tfrac{(q_t-1)n}{q_t}\}.
\]
Again, $\mathcal{R}$ is $\pi$-reducible for any basic path $\pi$ contained in $R_0$. Since $q_t$ divides $p_r$, the point $\frac{n}{q_t}$ belongs to $O_0\cap R_0$. Take $\pi:=\{0,\frac{n}{q_t}\}$.
Note that $\pi$ is a basic path in $\mathcal{P}$, $\mathcal{Q}$ and $\mathcal{R}$. Moreover, $\pi$ never splits in both $\mathcal{O}$ and $\mathcal{R}$. Since all inner points of $\mathcal{P}$ are inner points either in $\mathcal{O}$ or in $\mathcal{R}$, it follows that $\pi$ never splits in $\mathcal{P}$.
\end{proof}

\section{$k$-Flowers}\label{S8}
Following the sketch of the proof of Theorem~\ref{Tachin} outlined in Section~\ref{S5}, we have to deal now with the special case of patterns with only
one inner point and $k\ge2$ discrete components. When $k=2$, the following result (Theorem~5.2 of \cite{forward}) does the job.

\begin{theorem}\label{fwd}
Let $\mathcal{P}$ be an $n$-periodic pattern with two discrete components. If $h(\mathcal{P})>0$, then $h(\mathcal{P})\ge\log(\lambda_n)$.
\end{theorem}

For $k\ge3$, and in the spirit of the proof by induction outlined in Section~\ref{S5}, we need to relate our pattern of period $n$ with
another pattern with period less than $n$ and positive entropy. So, let $\mathcal{P}=([T,P],[f])$ be an $n$-periodic pattern. A pattern
$\mathcal{P}'$ will be said to be \emph{subordinated to} $\mathcal{P}$ if for some divisor $n>p>1$ of $n$ there is an $(n/p)$-periodic orbit
$P'\subset P$ of $f^p$ such that $\mathcal{P}'=([\chull{P'}_T,P'],[f^p\evalat{P'}])$. Clearly, this definition is independent
of the particular model $(T,P,f)$ representing $\mathcal{P}$.

The following result is Lemma~9.1 of \cite{ajm}. It allows us to estimate the entropy of a pattern from the entropy of a subordinated.

\begin{lemma}\label{subor}
Let $\mathcal{P}$ be an $n$-periodic pattern. Let $\mathcal{P}'$ be an $n'$-periodic pattern subordinated
to $\mathcal{P}$. If $h(\mathcal{P}')\ge\log(\lambda_{n'})$ then $h(\mathcal{P})>\log(\lambda_n)$.
\end{lemma}

A discrete component of a pattern will be said to be \emph{extremal} if it contains only one inner point. As an example, the discrete
components $A$, $B$ and $D$ are extremal for the pattern $\mathcal{P}$ shown in Figure~\ref{brutty2}.

Let $(T,P,f)$ be a model of a periodic pattern $\mathcal{P}$. Let $C$ be a discrete component of $(T,P)$.
We will say that a point $x\in C$ \emph{escapes from $C$} if $f(x)$ does not belong to the connected
component of $T\setminus \{x\}$ that intersects $\Int(\chull{C})$. Any discrete component $C$ of $(T,P)$ without
points escaping from it will be called a \emph{scrambled component} of $\mathcal{P}$. Clearly, this notion does
not depend on the particular chosen model of $\mathcal{P}$. So, it makes sense to say that the pattern
$\mathcal{P}$ \emph{has a scrambled component}. As an example, the point 7 escapes from $\{1,7,13\}$ in the
18-periodic pattern $\mathcal{P}_2$ shown in Figure~\ref{ullh0}, while does not scape from $C:=\{0,3,5,7,9,11,15,17\}$.
In fact, no point in $C$ escapes from $C$. So, $C$ is a scrambled component for $\mathcal{P}_2$. It is
easy to see that every periodic pattern has scrambled components (Lemma~4.2 of \cite{ajm}).

\begin{theorem}\label{1innerprevi}
Let $\mathcal{P}$ be an $n$-periodic pattern with positive entropy and at least three discrete components. Assume that
any opening of $\mathcal{P}$ has entropy zero. If $\mathcal{P}$ has an extremal scrambled component, then $\mathcal{P}$ has
subordinated patterns with positive entropy.
\end{theorem}
\begin{proof}
Let $(T,P,f)$ be a model of $\mathcal{P}$ and let $C$ be the extremal scrambled component of $\mathcal{P}$.
Then, there is only one inner point $x$ in $C$, and $f(x)\in C$ by definition of a scrambled component. Consider a sequence of openings
that joins together all discrete components different from $C$ into a single discrete component $D$, leading to a pattern $\mathcal{P}'$
with two discrete components, $C$ and $D$. Since $h(\mathcal{P'})=0$ by hypothesis, $\mathcal{P'}$ has a \emph{division} \cite{ajm}
with respect to $C$. In consequence, there exists $p\ge2$, a divisor of $n$, such that $f^i(D)\subset C$ for $1\le i<p$ and $f^p(D)=D$.
In other words, $\cup_{i=0}^{p-1} P_i$ is a $p$-block structure for $\mathcal{P}$, where $P_i:=f^i(D)$. Note that the blocks
$P_1,P_2,\ldots P_{p-1}$ are contained in $C$ and are, thus, trivial. Consider the pattern $\mathcal{Q}:=([\chull{P_0}_T,P_0],[f^p\evalat{P_0}])$.
Then, $\mathcal{Q}$ is subordinated to $\mathcal{P}$. Moreover, its entropy is positive, for otherwise the fact that all blocks but one
are trivial would easily imply that $h(\mathcal{P})=0$.
\end{proof}

When a pattern has only one inner point $x$, the discrete component containing the image of $x$ is clearly scrambled and extremal. So,
we have the next result as an immediate consequence of Theorem~\ref{1innerprevi}.

\begin{corollary}\label{1inner}
Let $\mathcal{P}$ be a positive entropy $k$-flower, with $k\ge3$. Assume that any opening of $\mathcal{P}$ has entropy zero.
Then, $\mathcal{P}$ has subordinated patterns with positive entropy.
\end{corollary}

\section{Triple chains}\label{S9}
The final stage in the proof of Theorem~\ref{Tachin} outlined in Section~\ref{S5} leaves us with the special case of a pattern
with exactly two inner points and three discrete components, a \emph{triple chain}. In order to find lower bounds for the entropy of
a triple chain $\mathcal{P}$, it is unavoidable to count coverings in the $\mathcal{P}$-path graph (equivalently, entries in the
path transition matrix). This section is devoted to this task. In our context, it is assumed that $\mathcal{P}$ is $\pi$-irreducible and
any of the two possible openings of $\mathcal{P}$ has entropy zero (property~(\ref{opening0})). Note that any opening of a triple chain has
two discrete components. So, to obtain a lower bound of the entropy of $\mathcal{P}$ we will proceed in two steps. First, we
will study the coverings in the path graph of zero entropy patterns with two discrete components. This is the aim of Lemmas~\ref{bifurcacions}
and \ref{cover-k}. Finally, we will study how the previous coverings, present in the two possible openings of the triple chain $\mathcal{P}$,
imply the existence of a number of coverings in the $\mathcal{P}$-path graph (Lemma~\ref{cover4}) that forces enough entropy for our purposes.

The results mentioned in the previous scheme are extremely technical. Readers are cautioned to follow the arguments using examples,
as the ones shown in the figures.

A basic path $\pi$ for a pattern with a separated structure of trivial blocks will be said to be \emph{in-block}
if it is contained in a block. Otherwise, it will be said to be \emph{inter-block}. As an example, $\{1,13\}$ is
an in-block basic path of $\mathcal{P}_2$ in Figure~\ref{ullh0}, while $\{0,15\}$ is inter-block.
The second statement of Proposition~\ref{twocharact} says that an in-block path never \emph{splits} (as defined in page~\pageref{pagesplit}).
On the other hand, next result states that inter-block basic paths do always split.

\begin{lemma}\label{inter-blocks}
Let $\mathcal{P}$ an $n$-periodic pattern with a separated structure of trivial blocks. Then any inter-block basic path of $\mathcal{P}$ splits before $n$ iterates.
\end{lemma}

\begin{proof}
The proof strongly relies on the construction of the maximal separated structure of trivial blocks in Proposition 9.5 of \cite{ajm} and its uniqueness (see Section~\ref{S3}).
The construction shows that if $\mathcal{P}$ is $\sigma$-reducible for a basic path $\sigma$, then each trivial block is obtained as the set of endpoints of a connected component
of $\cup_{i\ge0} \chull{f^i(\sigma)}$. In particular, $\sigma$ is contained in one block and is thus an in-block basic path. The uniqueness of the maximal structure of trivial
blocks implies that the same is true for any basic path $\sigma'\ne\sigma$ such that $\mathcal{P}$ is $\sigma'$-reducible. Now we note that if $\pi$ does not split in $n$ iterates, then
$\mathcal{P}$ is $\pi$-reducible. Then, by the previous discussion, $\pi$ has to be in-block, a contradiction.
\end{proof}

Let $\mathcal{P}$ be a non-trivial $n$-periodic pattern with entropy zero. By Proposition~\ref{56-comb}, $\mathcal{P}$ has a maximal structure of
trivial blocks and the corresponding combinatorial collapse $\mathcal{C}$ has entropy zero. Let $x$ be a point of $\mathcal{P}$. The point of
$\mathcal{C}$ corresponding to the collapse of the block containing $x$ will be denoted by $\overline{x}$, and this will be a standing notation
throughout this section. In fact, if $x$ is contained in the block $P_i$, then $\overline{x}$ is precisely the point of $\mathcal{C}$ labeled as $i$.
Let $\pi=\{x,y\}$ be an inter-block basic path of $\mathcal{P}$. Then, $\overline{x}\ne\overline{y}$. The binary set $\{\overline{x},\overline{y}\}$
will be denoted by $\overline{\pi}$. Note that, by property (b) of Definition~\ref{combicolla}, $\overline{\pi}$ is a basic path in $\mathcal{C}$.
As an example, consider the pattern $\mathcal{P}_2$ shown in Figure~\ref{ullh0}. The basic paths $\pi_1=\{11,8\}$ and $\pi_2=\{0,7\}$ are inter-block.
In this case, $\overline{\pi}_1=\{5,2\}$ and $\overline{\pi}_2=\{0,1\}$ are (respectively, in-block and inter-block) basic paths of the combinatorial
collapse $\mathcal{P}_1$.

The notation $\mathcal{O}$ for patterns of entropy zero used in the statements of this section suggests, as in Section~\ref{S7}, the term \emph{opening}.

\begin{lemma}\label{herencia}
Let $\mathcal{O}$ be a zero entropy periodic pattern and let $\mathcal{C}$ be the combinatorial collapse of $\mathcal{O}$.
Let $\pi$ be an inter-block basic path of $\mathcal{O}$. If $\overline{\pi}$ splits in $\ell$ iterates on $\mathcal{C}$, then $\pi$
splits in at most $\ell$ iterates on $\mathcal{O}$.
\end{lemma}
\begin{proof}
Consider any model $(T,P,f)$ of $\mathcal{O}$ and assume that $f^i(\pi)$ is a basic path for $0\le i<\ell$. From the definition of a block
structure it follows that, for all $0\le i<\ell$, the basic path $f^i(\pi)$ is inter-block in $\mathcal{O}$. Set $\{a,b\}:=f^\ell(\pi)$. By hypothesis,
$\overline{f^i(\pi)}$ is a basic path in $\mathcal{C}$ for $0\le i<\ell$, while $\overline{a}$ and $\overline{b}$ are separated by at least one inner point in $\mathcal{C}$.
We have to see that $a$ and $b$ are also separated in $\mathcal{O}$. Assume, by way of contradiction, that there exists a discrete component $D$ of
$\mathcal{O}$ containing $\{a,b\}$. In particular, the trivial blocks $K_a$ and $K_b$ in $\mathcal{O}$ whose collapse gives respectively the points
$\overline{a}$ and $\overline{b}$ of $\mathcal{C}$ satisfy $K_a\cap D\ne\emptyset$ and $K_b\cap D\ne\emptyset$. By definition of the combinatorial
collapse, this implies that there exists a single discrete component of $\mathcal{C}$ containing $\{\overline{a},\overline{b}\}$, a contradiction.
\end{proof}

Given a pattern $\mathcal{P}=([T,P],[f])$ and two basic paths $\pi$ and $\sigma$ of $\mathcal{P}$, we will say that $\pi$ is a \emph{strict pre-image of} $\sigma$
if there exists $j\ge1$ such that $f^i(\pi)$ is a basic path for $0\le i\le j$ and $f^j(\pi)=\sigma$. Note that, in this case, $f^i(\pi)$ are also strict
pre-images of $\sigma$ for $1\le i<j$.

The following result computes the number of iterations necessary for an inter-block basic path to split in a zero entropy pattern with two discrete components.
At this point we recover the notation introduced in Section~\ref{S2} and write $\{a,b\}\rightarrow \{c,d\}$ to indicate that the basic path $\{a,b\}$ $f$-covers the basic path $\{c,d\}$.

\begin{proposition}\label{bifurcacions}
Let $\mathcal{O}=([T,P],[f])$ be a zero entropy $n$-periodic pattern with two discrete components and a maximal structure of trivial blocks of cardinality $q$.
Let $\mathcal{C}$ be the corresponding combinatorial collapse. Assume that $\mathcal{O}$ is labeled in such a way that 0 is the unique inner point.
Let $\pi$ be an inter-block basic path $\pi$ of $\mathcal{O}$. Then,
\begin{enumerate}[(i)]
\item either splits in at most $\frac{n}{q}$ iterates,
\item or it is an strict pre-image of a basic path $\sigma=\{0,a+\frac{q-1}{q}n\}$ with $0<a<\frac{n}{q}$. In this case, $\overline{\sigma}$ is in-block in $\mathcal{C}$
and $\pi$ splits in at most $\frac{2n}{q}$ iterates.
\end{enumerate}
If in addition $\overline{\pi}$ is in-block in $\mathcal{C}$, then the following statements hold:
\begin{enumerate}[(a)]
\item If $\pi=\{0,a\}$ with $0<a<\frac{n}{q}$, then $\pi$ splits in $\frac{n}{q}-a$ iterates.
\item If $\pi=\{0,a+\tfrac{q-1}{q}n\}$ with $0<a<\frac{n}{q}$, then $\pi$ splits in $\frac{n}{q}$ iterates.
\end{enumerate}
\end{proposition}
\begin{proof}
Let $\{\mathcal{O}_i\}_{i=0}^s$ be the sequence of collapses of $\mathcal{O}$ according to Remark~\ref{explosions} and let $q_i$ be the cardinality of the blocks of $\mathcal{O}_i$. Since $\mathcal{O}$ is not trivial, $s\ge1$. The proof follows in two steps. First, we prove the result for a sequence of collapses of length $s=1$. Then we tackle the general case $s>1$ using the case $s=1$ on a particular subordinated pattern of $\mathcal{O}$.

Assume first that $s=1$ and let $q_0=n/q_1$ be the period of the combinatorial collapse $\mathcal{C}=\mathcal{O}_0$, which is a trivial pattern. The pattern $\mathcal{O}=\mathcal{O}_1$ is formed by $q_0=n/q_1$ trivial blocks of $q_1$ points. Let us denote by $P_i$, $i=0,\dots,q_0-1$, the trivial blocks of the pattern $\mathcal{O}$ according to the standing convention in Remark~\ref{standing2}. Notice that the block $P_0$, formed by the multiples of $q_0$ (mod $n$), is one of the two discrete components of $\mathcal{O}$.

Let $\pi$ be an inter-block basic path of $\mathcal{O}$. Since $\mathcal{C}$ is trivial, $\overline{\pi}$ is in-block in $\mathcal{C}$. The labeling of $\mathcal{O}$ is fixed by the unique inner point. So, we can write the points in $P_i$ as $i+\ell q_0$ with $0\leq\ell\leq q_1-1$. The point $i+(q_1-1)q_0$ will be called the last point in $P_i$. The inter-block basic path $\pi$ connects a point of the block $P_i$ with a point of the block $P_j$. Along the proof we consider that the blocks are ordered in such a way that $0\leq i<j\leq q_0-1$. We distinguish three types of inter-block basic paths of $\mathcal{O}$ depending on the points that are connected.

\begin{figure}
\centering
\includegraphics[scale=0.65]{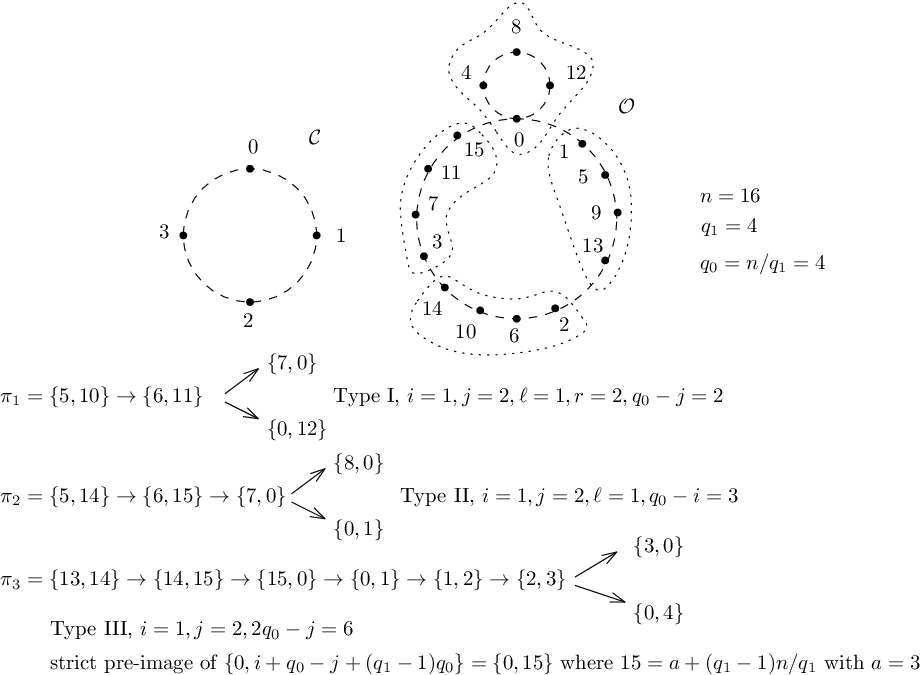}
\caption[fck]{The three types of paths in Proposition~\ref{bifurcacions}.\label{exemple1explosio}}
\end{figure}

Type I. $\pi$ connects any point of $P_i$ with one of $P_j$ that is not the last one, $0\leq i<j\leq q_0-1$. In this situation, we can write $\pi=\{i+\ell q_0,j+rq_0\}$ with $0\leq \ell\leq q_1-1$ and $0\leq r\leq q_1-2$.

Type II. $\pi$ connects a point of the block $P_i$ that is not the last one with the last point of $P_j$, $0\leq i<j\leq q_0-1$. In this case, $\pi=\{i+\ell q_0, j+(q_1-1)q_0\}$ with $0\leq \ell\leq q_1-2$.

Type III. $\pi$ connects the last points of the blocks $P_i$ and $P_j$, $1\leq i<j\leq q_0-1$. In this latter case, $\pi=\{i+(q_1-1)q_0,j+(q_1-1)q_0\}$.

Since $\pi$ is an inter-block basic path, if $i=0$ for Type I and II then $\ell=0$. That is, only the point $0\in P_0$ can be connected to a point of a different trivial block. For this reason, $i\geq 1$ in Type III.

Notice that an inter-block basic path $\pi=\{a,b\}$ splits in $k$ iterates if $k$ is the smallest integer such that either $a+k$ or $b+k$ is a multiple of $q_0$ different from $0$ (mod $n$). Indeed, since $\pi$ is inter-block, $a+k$ and $b+k$ cannot be both multiple of $q_0$. Otherwise, $f^k(\pi)$ is a basic path joining two points of $P_0$ and is, therefore, in-block. Since $0$ is the only inner point and $P_0$ is a whole discrete component, the previous condition implies that $a+k$ and $b+k$ are on different discrete components. On account of the previous, now we compute the iterates that an inter-block basic path of each type requires to split.

If $\pi$ is of Type I, then it splits in $q_0-j$ iterates:
\[ \pi \rightarrow \overset{q_0-j)}{\cdots} \rightarrow \{i+q_0-j+\ell q_0,0\}\cup\{0,(r+1)q_0\}. \]
Indeed, since $0\leq i<j\leq q_0-1$ then $j+rq_0$ reaches the point $(r+1)q_0$ in $q_0-j$ iterates, whereas $i+\ell q_0$ needs $q_0-i>q_0-j$ iterates to reach a multiple of $q_0$. Since $0\leq r\leq q_1-2$ then $(r+1)q_0\neq 0$ (mod $n$), so the splitting occurs.

If $\pi$ is of Type II, then it splits in $q_0-i$ iterates:
\[ \pi \rightarrow \overset{q_0-i)}{\cdots} \rightarrow \{(\ell+1)q_0,0\}\cup\{0,j-i\}. \]
Indeed, in this case, although $i<j$ and $j+(q_1-1)q_0$ reaches a multiple of $q_0$ in $q_0-j$ iterates, the multiple is $q_1q_0=n=0$ (mod $n$). Therefore, there is no splitting in $q_0-j$ iterates. On the other hand, in $q_0-i$ iterates the splitting occurs.

Finally, if $\pi$ is of Type III, then it splits in $2q_0-j$ iterates. Indeed, in $q_0-j$ iterates:
\[ \pi \rightarrow \overset{q_0-j)}{\cdots} \rightarrow \{i+q_0-j+(q_1-1)q_0,0\}\]
The basic path $\{i+q_0-j+(q_1-1)q_0,0\}$ is of Type II with $i=0$. So, it splits in $q_0$ iterates. Summing up, $\pi$ splits in $2q_0-j$ iterates.

The previous discussion proves the result for $s=1$. Indeed, every inter-block basic path splits in at most $q_0=\frac{n}{q_1}$ iterates with the exception of the strict pre-images of $\{0,a+(q_1-1)q_0\}$ with $0<a=i+q_0-j<q_0$, which split in $2q_0-j<\frac{2n}{q_1}$ iterates. Moreover, the case (a) corresponds to a Type I basic path by taking $i=\ell=0$, $j=a$ and $r=0$, so $\pi$ splits in $q_0-a=n/q_1-a$ iterates.
Taking $i=\ell=0$ and $j=a$ on a Type II basic path, $\pi=\{0,a+(q_1-1)q_0\}$ splits in $q_0=n/q_1$ iterates, proving (b). In Figure~\ref{exemple1explosio} we show examples of each type for $n=16$ and $q_0=4$.

\begin{figure}
\centering
\includegraphics[scale=0.65]{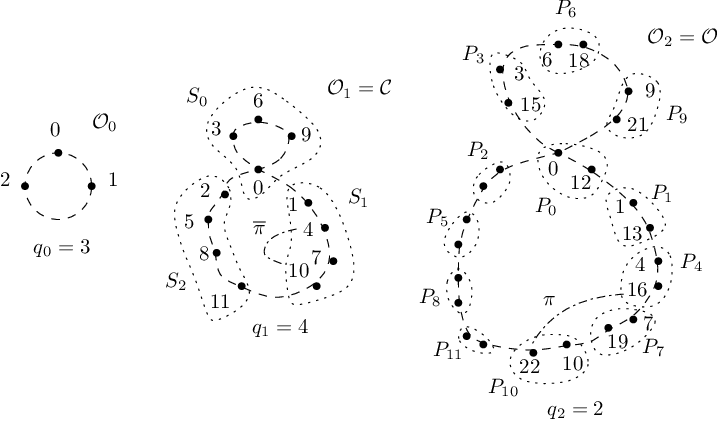}
\caption[fck]{Notation in the proof of Proposition~\ref{bifurcacions}.\label{nova}}
\end{figure}

Assume now that the sequence of collapses of $\mathcal{O}$ has length $s>1$. Set $\mathcal{C}:=\mathcal{O}_{s-1}$ and $\mathcal{O}:=\mathcal{O}_s$. Moreover, each pattern $\mathcal{O}_i$ for $1\le i\le s$ has a unique inner point, labeled as $0$ according to Remark~\ref{standing3}. Let $q_0$ be the period of $\mathcal{O}_0$ and, for $1\le i\le n$, let $q_i$ be the cardinality of the blocks of $\mathcal{O}_i$. Then, $n=\prod_{i=0}^s q_i$. According to the notation in the statement, $q_s=q$. The pattern $\mathcal{O}$ has a maximal structure of $n/q$ trivial separated blocks of $q$ points, and the pattern $\mathcal{C}$ has a maximal structure of $n/(q_{s-1}q)$ trivial blocks of cardinality $q_{s-1}$. Let us denote by $S_i$ and $P_i$ the blocks of the patterns $\mathcal{C}$ and $\mathcal{O}$, respectively. Since 0 is the unique inner point in
$\mathcal{O}$, from Lemma~\ref{critic} it follows that 0 is bidirectional. Then, the trivial block $P_0$ of $\mathcal{O}$ and the trivial block $S_0$ of $\mathcal{C}$ are contained
in different 0-branches. See Figure~\ref{nova} for an example with $s=2$, $q_0=3$, $q_1=4$, $q=q_2=2$, $n=24$.

Let $\pi$ be an inter-block basic path of $\mathcal{O}$. If $\overline{\pi}$ is inter-block in $\mathcal{C}$, by Lemma~\ref{inter-blocks} $\overline{\pi}$ splits before $n/q$ iterates.
By Lemma~\ref{herencia}, this property is inherited by $\pi$, which splits in at most $n/q$ iterates, as desired.

Let us assume now that $\pi$ is an inter-block basic path of $\mathcal{O}$ such that $\overline{\pi}$ is in-block in $\mathcal{C}$. That is, $\overline{\pi}$ is contained in
$S_{\eta}$ for some $\eta\in\{0,\dots,\frac{n}{q_{s-1}q}-1\}$. Note that all points in a block of $\mathcal{C}$ differ by a multiple of $n/(q_{s-1}q)$, while all points in a block
of $\mathcal{O}$ differ by a multiple of $n/q$. It follows that $\overline{\pi}$ has the form
\[ \overline{\pi}=\{\overline{a},\overline{b}\}=\{\eta+i\tfrac{n}{q_{s-1}q},\eta+j\tfrac{n}{q_{s-1}q}\} \]
with $0\le\eta\leq \frac{n}{q_{s-1}q}-1$ and $0\le i<j\leq q_{s-1}-1$, while $\pi$ has the form
\[ \pi=\{a,b\}=\{\overline{a}+\ell\tfrac{n}{q},\overline{b}+r\tfrac{n}{q}\} \]
with $0\le\ell,r\leq q-1$. For the sake of intuition, note that $\eta$ labels the block $S_{\eta}$ of $\mathcal{C}$ containing $\overline{\pi}$, while the blocks of $\mathcal{O}$
containing $a$ and $b$ are, respectively, $P_{\overline{a}}$ and $P_{\overline{b}}$. Going back to the example shown in Figure~\ref{nova}, if we take $\pi=\{16,22\}$, then
$\overline{\pi}=\{4,10\}$, $\eta=1$, $i=1$, $j=3$, $\overline{a}=4$, $\overline{b}=10$, $\ell=r=1$.

Let us study the iterates $f^i(\pi)$. Since 0 is the unique inner point in $\mathcal{O}$, a pair $\{x,y\}$ of points of $\mathcal{O}$ is not a basic path if and only if 0 separates
$x$ and $y$. Since $\overline{\pi}$ is in-block in $\mathcal{C}$, it never splits by Proposition~\ref{twocharact}. Moreover, $0\in P_0$. It follows that a basic path $\pi$ may split in $k$ iterates only if $\overline{f^k(\pi)}\subset S_0$. Let $\pi_0$ be the first iterate of $\pi$ such that $\overline{\pi}_0\subset S_0$.
Then,
\[
\pi_0=\{\tfrac{n}{q_{s-1}q}(i+\ell q_{s-1}),\tfrac{n}{q_{s-1}q}(j+rq_{s-1})\}
\]
for some $0\le i<j\leq q_{s-1}-1$ and $0\le\ell,r\leq q-1$. Let us look at the worst-case scenario by assuming that $\pi_0$ is a basic path.
That is to say, we have the sequence of non-splitting coverings
\[
\pi \rightarrow f(\pi) \rightarrow f^2(\pi) \rightarrow \overset{\frac{n}{q_{s-1}q}-\eta)}{\cdots}  \rightarrow \pi_0.
\]

In order to bound the number of iterates required by $\pi$ to split, we study $\pi_0$. Let us consider the subordinated pattern $\mathcal{O}':=([\chull{P_0}_T,P_0],[f^{\frac{n}{q_{s-1}q}}])$. Note that $\mathcal{O}'$ has two discrete components, entropy zero and a maximal structure of $q_{s-1}$ trivial blocks given by
\[ P_0\cup P_{\frac{n}{q_{s-1}q}}\cup\ldots\cup P_{\frac{(q_{s-1}-1)n}{q_{s-1}q}}. \]

Moreover, the corresponding combinatorial collapse $\mathcal{C}'$ is a trivial pattern of period $q_{s-1}$. In other words, the sequence of collapses of $\mathcal{O}'$ reduces to $\{\mathcal{C}',\mathcal{O}'\}$ and thus we can apply the discussion about types of basic paths and coverings used in the case $s=1$. Let us take the labeling of $\mathcal{O}'$ such that the only inner point reads as 0. See Figure~\ref{exemple2explosio} for a picture of the patterns $\mathcal{C}'$ and $\mathcal{O}'$ corresponding to the example shown in Figure~\ref{nova}.

\begin{figure}
\centering
\includegraphics[scale=0.65]{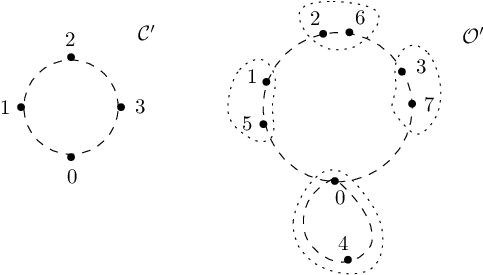}
\caption[fck]{The subordinated pattern $\mathcal{O}'$ and its collapse $\mathcal{C}'$ for the example shown in Figure~\ref{nova}.\label{exemple2explosio}}
\end{figure}

Notice that there is a correspondence between the basic path $\pi_0$ in $\mathcal{O}$ and the basic path $\{i+\ell q_{s-1},j+rq_{s-1}\}$ in $\mathcal{O}'$. Since $\pi_0$ may only split when $\overline{\pi}_0$ returns to $S_0$, it suffices to study the number of iterates required by $\{i+\ell q_{s-1},j+rq_{s-1}\}$ to split in $\mathcal{O}'$ and then multiply the length of the sequence of paths by $\frac{n}{q_{s-1}q}$. As it was stated in the discussion of the case $s=1$, we have three situations depending on the type of path.

If $\{i+\ell q_{s-1},j+rq_{s-1}\}$ in $\mathcal{O}'$ is of Type I then $\pi_0$ splits in $\frac{n}{q_{s-1}q}(q_{s-1}-j)=\frac{n}{q}-\frac{j}{q_{s-1}q}n$ iterates in $\mathcal{O}$. Taking $i=\ell=r=0$ and $a=\frac{j}{q_{s-1}q}n$, this proves (a).

If $\{i+\ell q_{s-1},j+rq_{s-1}\}$ in $\mathcal{O}'$ is of Type II then $r=q-1$ and $\pi_0$ splits in $\frac{n}{q_{s-1}q}(q_{s-1}-i)=\frac{n}{q}-\frac{i}{q_{s-1}q}n$ iterates in $\mathcal{O}$. Taking $i=\ell=0$ and $a=\frac{j}{q_{s-1}q}n$, this proves (b).

Lastly, if $\{i+\ell q_{s-1},j+rq_{s-1}\}$ in $\mathcal{O}'$ is of Type III then $\ell=r=q-1$ and $\pi_0$ splits in $\frac{n}{q_{s-1}q}(2q_{s-1}-j)=\frac{2n}{q}-\frac{j}{q_{s-1}q}n$ iterates in $\mathcal{O}$ and it is an strict pre-image of the basic path $\{0,\frac{n}{q_{s-1}q}(i+q_{s-1}-j+(q-1)q_{s-1})\}$.

The previous holds for $\pi_0$. In order to bound the iterates required by $\pi$ to split we add $\frac{n}{q_{s-1}q}-\eta$ to the previous. So, depending on the types before, for $1\le\eta\le\frac{n}{q_{s-1}q}-1$, either
\begin{itemize}
\item $\pi$ splits in $\frac{n}{q}-\frac{j-1}{q_{s-1}q}n-\eta$ iterates, or
\item $\pi$ splits in $\frac{n}{q}-\frac{i-1}{q_{s-1}q}n-\eta$ iterates, or
\item $\pi$ splits in $\frac{2n}{q}-\frac{j-1}{q_{s-1}q}n-\eta$ iterates and it is an strict pre-image of \[ \{0,\tfrac{n}{q_{s-1}q}(i+q_{s-1}-j+(q-1)q_{s-1})\}. \]
\end{itemize}
This proves that every inter-block basic path of $\mathcal{O}$ splits after at most $2n/q$ iterates. Moreover, the only inter-block basic paths splitting in more than $n/q$ iterates are strict pre-images of some $\{0,a+\frac{q-1}{q}n\}$, with $0<a=\frac{n}{q_{s-1}q}(i+q_{s-1}-j)<\frac{n}{q}$, proving the result.
\end{proof}

The previous result states that almost every inter-block basic path of a zero entropy pattern with two discrete components splits in at most $n/q$ iterates with
the exception of those considered in (ii). The following results are concerned with the bound for the latter case. The first result states that the ``time reverse''
of a zero entropy pattern $\mathcal{O}$ with two discrete components coincides with $\mathcal{O}$. Figure~\ref{reverse} shows an example that illustrates this remarkable
property, that is not true for general zero entropy patterns. It is possible to prove it using sequences of collapses and Proposition~\ref{56}, but we use
a result from \cite{forward} to get a considerably shorter proof.

\begin{figure}
\centering
\includegraphics[scale=0.65]{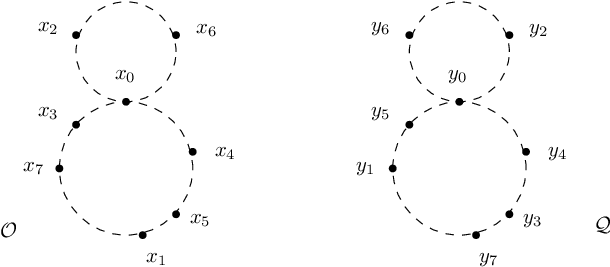}
\caption[fck]{A pattern $\mathcal{O}$ and its time reverse $\mathcal{Q}$ as defined in Lemma~\ref{preimatges}.\label{reverse}}
\end{figure}

\begin{lemma}\label{preimatges}
Let $(T,P,f)$ be the canonical model of an $n$-periodic pattern $\mathcal{O}$ with entropy zero and two discrete components. Let $P=\{x_i\}_{i=0}^{n-1}$ be time labeled. Consider
the relabeling of $P$ given by $y_i:=x_{n-i\bmod n}$ and the map $\map{g}{P}$ defined by $g(y_i):=y_{i+1\bmod n}$ for $0\le i<n$. Then, $([T,P],[g])=\mathcal{O}$.
\end{lemma}
\begin{proof}
Assume without loss of generality that $x_0$ is the unique inner point of $\mathcal{O}$. From the definitions we get that $P$ is an $n$-periodic orbit of $g$, time labeled
as $P=\{y_i\}_{i=0}^{p-1}$. Thus, $([T,P],[g])$ is an $n$-periodic pattern $\mathcal{Q}$. By definition, $y_0=x_0$, so that $y_0$ is the only inner
point of $\mathcal{Q}$. To see that $\mathcal{O}=\mathcal{Q}$ we have to show that both patterns have the same discrete components.

For any tree map $\map{F}{S}$, an ordered set $(a,b,c)$ of three points of $S$ is called a \emph{forward triplet of $F$} if $b\in(a,c)$, $f(a)=b$, $f(b)=c$,
and $\{a,b,c\}$ is contained in a periodic orbit of $F$. By Theorem~1.1 of \cite{forward}, $F$ has positive entropy if and only if there exists $k\ge1$ such that $F^k$
has a forward triplet. Thus, since $h(f)=h(\mathcal{O})=0$, $f$ cannot have forward triplets. It easily follows that both $x_i$ and $x_{n-i}$ belong to the same discrete
component of $\mathcal{O}$ for all $1\le i<n$. But $\{x_i,x_{n-i}\}=\{y_{n-i},y_i\}$, implying that both $\mathcal{O}$ and $\mathcal{Q}$ have exactly the same discrete
components.
\end{proof}

\begin{lemma}\label{pre-maximal}
The basic path $\sigma=\{0,a+\frac{q-1}{q}n\}$ with $0<a<\frac{n}{q}$ in (ii) of Proposition~\ref{bifurcacions} has at most $a-1$ strict pre-images. Moreover, a basic path
$\pi=\{0,y\}$ cannot be an strict pre-image of $\sigma$.
\end{lemma}
\begin{proof}
By Lemma~\ref{preimatges}, the pattern $\mathcal{O}$ coincides with its time reverse. In particular, the basic path $\{0,a+\frac{q-1}{q}n\}$ has as many pre-images as basic paths are covered by $\{0,\frac{n}{q}-a\}$ before splitting. The basic path $\sigma$ is inter-block and $\overline{\sigma}$ is in-block in the corresponding combinatorial collapse $\mathcal{C}$ of $\mathcal{O}$. Therefore, the same is true for the basic path $\{0,\frac{n}{q}-a\}$. Since $0<\frac{n}{q}-a<\frac{n}{q}$, by Proposition~\ref{bifurcacions} (a), the basic path $\{0,\frac{n}{q}-a\}$
splits in $\frac{n}{q}-\bigl(\frac{n}{q}-a\bigr)=a$ iterates. This proves the first assertion of the lemma.

The second assertion, using the time reverse property, is equivalent to show that the basic path $\{0,n-y\}$ is not covered by $\{0,\frac{n}{q}-a\}$ before splitting. That is, before $a$ iterates. This is clear, since neither $0$ nor $\frac{n}{q}-a$ map on $0$ before $a$ iterates.
\end{proof}

Now we can use Lemma~\ref{pre-maximal} together with Proposition~\ref{bifurcacions} to find the desired coverings.

\begin{lemma}\label{cover-k}
Let $\mathcal{O}$ be a zero entropy $n$-periodic pattern with two discrete components and a maximal structure of trivial blocks of cardinality $q$.
If $q\geq 3$ then any inter-block basic path of $\mathcal{O}$ covers at least four basic paths in $n$ iterates.
\end{lemma}
\begin{proof}
Let us label $\mathcal{O}$ in such a way that $0$ is the unique inner point and let $\pi$ be an inter-block basic path of $\mathcal{O}$. By Proposition~\ref{bifurcacions},
\begin{enumerate}[(i)]
\item either $\pi$ splits in at most $\frac{n}{q}$ iterates,
\item or $\pi$ is a strict pre-image of a basic path $\{0,a+\frac{q-1}{q}n\}$ with $0<a<\frac{n}{q}$ which splits in $\frac{n}{q}$ iterates.
\end{enumerate}

In the case (i), $\pi$ covers two basic paths $\{0,y\}$ and $\{0,z\}$ before $\frac{n}{q}$ iterates. Notice that both $\{0,y\}$ and $\{0,z\}$ cannot be in-block basic paths of $\mathcal{O}$. Otherwise, since $0$ is inner of $\mathcal{O}$ and $y$ and $z$ are contained in different discrete components, the trivial block that contains $0$ would contain points of two different discrete components, a contradiction. Therefore, we can assume $\{0,y\}$ to be an inter-block basic path of $\mathcal{O}$. Moreover, by Lemma~\ref{pre-maximal}, an inter-block basic path of the form $\{0,y\}$ cannot be an strict pre-image of a basic path of the form $\{0,a+\frac{q-1}{q}n\}$ with $0<a<\frac{n}{q}$. Therefore, again by Proposition~\ref{bifurcacions}, $\{0,y\}$ splits in at most $\frac{n}{q}$ iterates, covering two basic paths $\{0,y_1\}$ and $\{0,y_2\}$. Again one of them is inter-block of $\mathcal{O}$ and splits in at most $\frac{n}{q}$ iterates. Therefore, $\pi$ covers at least four basic paths in $\frac{3n}{q}\leq n$ iterates. This proves the result in the case (i).

In the case (ii), by Lemma~\ref{pre-maximal}, $\pi$ covers $\{0,a+\frac{q-1}{q}n\}$ in at most $a-1$ iterates and $\{0,a+\frac{q-1}{q}n\}$ covers $\{0,a\}$ and $\{0,\frac{n}{q}\}$ in $\frac{n}{q}$ iterates. By Proposition~\ref{bifurcacions}(a), the basic path $\{0,a\}$ splits and covers two basic paths $\{0,u\}$ and $\{0,v\}$ in $\frac{n}{q}-a$ iterates and, since one must be inter-block in $\mathcal{O}$, again splits in at most $\frac{n}{q}$ iterates as shown before. Therefore, $\pi$ covers at least four basic paths in $a-1+\frac{n}{q}+\frac{n}{q}-a+\frac{n}{q}=\frac{3n}{q}-1<n$ iterates, proving the result in the case (ii).
\end{proof}

\begin{remark}\label{Ocovers}
Let $\mathcal{P}$ be a pattern and let $\mathcal{O}$ be an opening of $\mathcal{P}$. Let $\pi$ be a basic path of $\mathcal{P}$. Then $\pi$ is also a basic path of $\mathcal{O}$. Moreover, if $\pi$ covers $k$ basic paths in $\ell$ iterates in $\mathcal{O}$ then $\pi$ covers at least $k$ basic paths in $\ell$ iterates in $\mathcal{P}$.
\end{remark}

By collecting all previous results, finally we get the desired lower bound for coverings in a triple chain.

\begin{figure}
\centering
\includegraphics[scale=0.65]{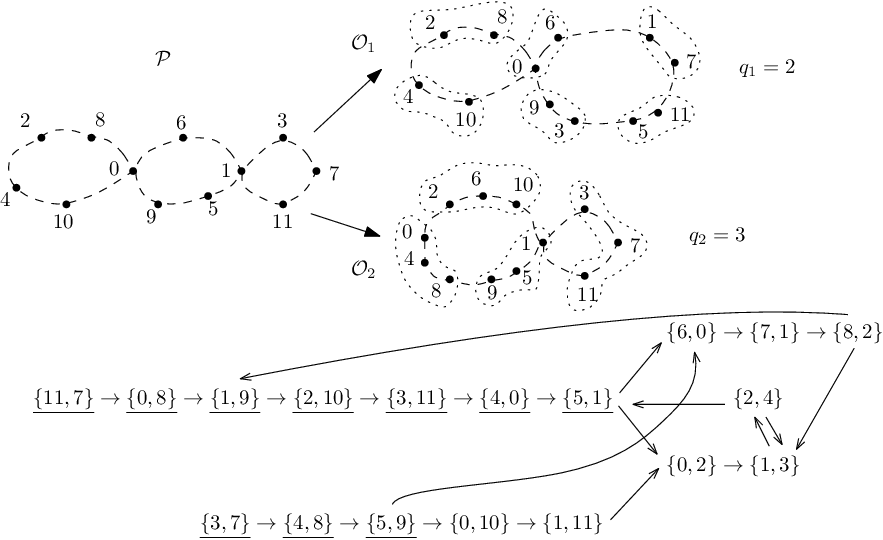}
\caption[fck]{An illustration of the proof of Proposition~\ref{cover4}. Some loops of the $\mathcal{P}$-path graph obtained in the proof are shown. The underlined basic paths are in-block in $\mathcal{O}_2$.\label{opening23}}
\end{figure}

\begin{proposition}\label{cover4}
Let $\mathcal{P}$ be an $n$-periodic $\pi$-irreducible triple chain. Assume that the two possible openings $\mathcal{O}_{1}$ and $\mathcal{O}_{2}$ of $\mathcal{P}$ have entropy zero.
Then, any basic path $\pi$ of $\mathcal{P}$ covers at least four basic paths in $n$ iterates.
\end{proposition}
\begin{proof}
By Remark~\ref{Ocovers} a basic path $\pi$ of $\mathcal{P}$ is also a basic path of both $\mathcal{O}_1$ and $\mathcal{O}_2$. We claim that $\pi$ is inter-block for some $\mathcal{O}_i$. Indeed, if $\pi$ is in-block in both $\mathcal{O}_1$ and $\mathcal{O}_2$, then $\pi$ does not split through any of the two inner points of $\mathcal{P}$. Consequently, $\pi$ never splits in $\mathcal{P}$ and so $\mathcal{P}$ is $\pi$-reducible, a contradiction.

The patterns $\mathcal{O}_i$, $i=1,2$, have zero entropy. So, by Proposition~\ref{56}, each of them has a maximal structure of trivial blocks of cardinality $q_i\geq 2$.

Let us first prove the result when $q_i\geq 3$. As stated before, $\pi$ is inter-block for some of the openings, let us say $\mathcal{O}_1$ without loss of generality. Since $q_1\geq 3$, by Lemma~\ref{cover-k}, $\pi$ covers at least four basic paths in $n$ iterates in $\mathcal{O}_1$. By Remark~\ref{Ocovers}, this property is inherited in $\mathcal{P}$, so $\pi$ covers at least four basic paths in $n$ iterates in $\mathcal{P}$. This proves the result in the first situation.

Now assume that $q_1=q_2=2$. In this case, the basic path $\{0,\frac{n}{2}\}$ is in-block in both $\mathcal{O}_1$ and $\mathcal{O}_2$. By the discussion at the beginning of the proof, this produces contradiction with the $\pi$-irreducibility of $\mathcal{P}$.

We are left with the case $q_1=2$ and $q_2\geq 3$. Again, $\pi$ is inter-block in $\mathcal{O}_1$ or $\mathcal{O}_2$. If $\pi$ is inter-block in $\mathcal{O}_2$, the result follows as in the first case since $q_2\geq 3$. So, we can assume that $\pi$ is in-block in $\mathcal{O}_2$ and, in consequence, inter-block in $\mathcal{O}_1$.

Let us relabel $\mathcal{P}$ and, accordingly, the openings $\mathcal{O}_i$, in such a way that the inner point of $\mathcal{O}_1$ is $0$. We denote by $j$ the inner point of $\mathcal{O}_2$. The basic path $\pi$ is in-block in $\mathcal{O}_2$, so in $\mathcal{P}$ the first splitting is through the inner $0$.
Since $\pi$ is inter-block in $\mathcal{O}_1$, by Proposition~\ref{bifurcacions}, one of the following situations occurs in $\mathcal{O}_1$:
\begin{enumerate}[(i)]
	\item either $\pi$ splits in at most $\frac{n}{2}$ iterates,
	\item or $\pi$ is a strict pre-image of a basic path $\{0,a+\frac{n}{2}\}$ with $0<a<\frac{n}{2}$, which splits in $\frac{n}{2}$ iterates.
\end{enumerate}

In both cases $\pi$ covers two basic paths in $\mathcal{O}_1$ after the first splitting. Since $\mathcal{P}$ is a triple chain, at least one of such paths is also a basic path in $\mathcal{P}$. For the sake of brevity, we will focus on the worst scenario which corresponds to assuming that the two basic paths covered in $\mathcal{O}_1$ are also basic paths in $\mathcal{P}$. The reader may easily check that if this is not the case, then a third basic path is covered in $\mathcal{P}$ during the first splitting, and the upper bounds obtained below are valid for the basic path shared between $\mathcal{O}_1$ and $\mathcal{P}$.

Consider the case (i). Since $0$ is the inner point of $\mathcal{O}_1$, then $\pi$ covers in $\mathcal{O}_1$ two basic paths $\{0,y\}$ and $\{0,z\}$ in at most $\frac{n}{2}$ iterates. As noticed above, we are assuming that both $\{0,y\}$ and $\{0,z\}$ are basic paths in $\mathcal{P}$. Clearly, $y\neq z$ and so we can assume also $z\neq \frac{n}{2}$. Consequently, $\{0,z\}$ is an inter-block in $\mathcal{O}_1$. Moreover, by Lemma~\ref{pre-maximal}, $\{0,z\}$ is not an strict pre-image of a basic path of the form $\{0,a+\frac{n}{2}\}$. Then, by Proposition~\ref{bifurcacions}, $\{0,z\}$ splits in at most $\frac{n}{2}$ iterates covering two basic paths $\{0,z_1\}$ and $\{0,z_2\}$. Since $\{0,z\}$ is a basic path in $\mathcal{P}$, then $\{0,z\}$ covers at least two basic paths in $\frac{n}{2}$ iterates in $\mathcal{P}$. Now we have two cases depending on the value of $y$. If $y\neq\frac{n}{2}$ the same argument applies for $\{0,y\}$ and, summing up, $\pi$ covers at least four basic paths in $n$ iterates in $\mathcal{P}$, proving the result in this case. The following diagram illustrates the coverings in this first situation inside case (i).
\[
\includegraphics[scale=1]{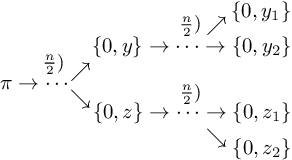}
\]

If $y=\frac{n}{2}$ then $\{0,\frac{n}{2}\}$ is in-block in $\mathcal{O}_1$. Since $\{0,\frac{n}{2}\}$ is a basic path in $\mathcal{P}$, it is also a basic path in $\mathcal{O}_2$. Moreover, it must be inter-block. By Proposition~\ref{bifurcacions}, either $\{0,\frac{n}{2}\}$ covers two basic paths before $\frac{n}{q_2}$ iterates or it is a strict pre-image of a basic path $\{j,j+b+\frac{(q_2-1)n}{q_2}\}$, where $0<b<\frac{n}{q_2}$ and $j$ is the inner point of $\mathcal{O}_2$. The second alternative, however, cannot be satisfied. Indeed, the time distance between the two points of an iterate of a basic path is conserved while there is no splitting. If $\{0,\frac{n}{2}\}$ is a strict pre-image of $\{j,j+b+\frac{(q_2-1)n}{q_2}\}$, then the distance should be conserved, but $b+\frac{(q_2-1)n}{q_2}\geq \frac{n}{2}$. Therefore, $\{0,\frac{n}{2}\}$ covers two basic paths in $\mathcal{O}_2$ in at most $\frac{n}{q_2}$ iterates. Since $q_2\geq 3$ then, summing up, $\pi$ covers at least four basic paths in $n$ iterates in $\mathcal{P}$, proving the result for the case (i). The following diagram illustrates the coverings in this second situation inside case (i).
\[
\includegraphics[scale=1]{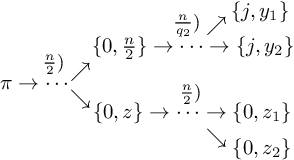}
\]

The basic paths $\{0,8\}$ and $\{3,7\}$ in Figure~\ref{opening23} are examples of maximal length of case (i). The basic path $\{0,8\}$ splits in $\frac{n}{2}=6$ iterates and covers $\{0,y\}=\{0,6\}$ and $\{0,z\}=\{0,2\}$. The path $\{0,6\}$ is of the form $\{0,\frac{n}{2}\}$, so it is in-block in $\mathcal{O}_1$ and inter-block in $\mathcal{O}_2$. It splits in $3<\frac{n}{2}=6$ iterates. The path $\{0,2\}$ is inter-block in both $\mathcal{O}_1$ and $\mathcal{O}_2$ and splits in $2<\frac{n}{2}=6$ iterates. A similar phenomenon occurs for $\{3,7\}$.

Let us now consider the case (ii). By Lemma~\ref{preimatges} the basic path $\{0,a+\frac{n}{2}\}$ has, at most, $a-1$ strict pre-images. Thus, $\pi$ covers $\{0,a+\frac{n}{2}\}$ in at most $a-1$ iterates. Since $\pi$ is in-block in $\mathcal{O}_2$, $\{0,a+\frac{n}{2}\}$ must also be an in-block path of $\mathcal{O}_2$. Hence, $a+\frac{n}{2}=k\frac{n}{q_2}$ for some $1\le k\le q_2-1$.
Moreover, $\{0,a+\frac{n}{2}\}$ splits in $\frac{n}{2}$ iterates and covers the basic paths $\{0,a\}$ and $\{0,\frac{n}{2}\}$. Recall that we are assuming that both $\{0,a\}$ and $\{0,\frac{n}{2}\}$ are basic paths in $\mathcal{P}$ and so in $\mathcal{O}_2$.
The basic path $\{0,\frac{n}{2}\}$ is inter-block in $\mathcal{O}_2$ and, as proved in case (i), covers two basic paths before $\frac{n}{q_2}$ iterates. On the other hand, $\{0,a\}$ is inter-block for $\mathcal{O}_1$ and, since $a<\frac{n}{2}$, it covers two basic paths in $\frac{n}{2}-a$ iterates by Proposition~\ref{bifurcacions}(a). Summing up, $\pi$ covers two basic paths in $a-1+\frac{n}{2}+\frac{n}{q_2}=(k+1)\frac{n}{q_2}-1\leq n-1$ iterates through $\{0,\frac{n}{2}\}$ and two basic paths in $a-1+\frac{n}{2}+\frac{n}{2}-a=n-1$ iterates through $\{0,a\}$, which proves that $\pi$ covers at least four basic paths in $n$ iterates. The following diagram illustrates the coverings in case (ii).
\[
\includegraphics[scale=1]{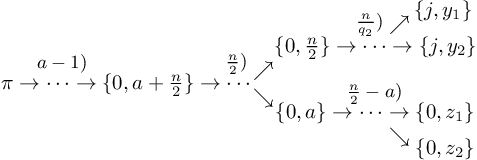}
\]

The basic path $\{11,7\}$ is the only one satisfying case (ii) in Figure~\ref{opening23}. Here $\{0,a+\frac{n}{2}\}=\{0,8\}$ with $a=2$. Indeed, $\{0,8\}$ has at most $a-1=1$ pre-images and $\{0,8\}$ splits exactly in $\frac{n}{2}=6$ iterates covering $\{0,a\}=\{0,2\}$ and $\{0,\frac{n}{2}\}=\{0,6\}$.
\end{proof}

Let $A=(a_{ij})$ be an $n\times n$ nonnegative matrix. Recall that $\rho(A)$ stands for the spectral radius of $A$. For $1\leq i\leq n$, let $r_i(A)=\sum_{j=1}^n a_{ij}$ be the $i$-th row sum of $A$. The following result is well-known \cite{Minc}.

\begin{theorem}\label{spectral}
If $A$ is a nonnegative matrix then
\[ \min_{1 \leq i \leq n} r_i(A)\leq \rho(A) \leq \max_{1 \leq i \leq n} r_i(A). \]
\end{theorem}

\begin{corollary}\label{n4}
Let $\mathcal{P}$ be an $n$-periodic and $\pi$-irreducible triple chain. Assume that the two possible openings of $\mathcal{P}$ have entropy zero. Then, $h(\mathcal{P})>\log(\sqrt[n]{4})$.
\end{corollary}
\begin{proof}
By Remark~\ref{patent}, $h(\mathcal{P})=\log\max\{\rho(M),1\}$, where $M$ is the path transition matrix of $\mathcal{P}$.
By Proposition~\ref{cover4}, any basic path of $\mathcal{P}$ covers at least four basic paths in $n$ iterates. In particular, the sum of the elements on each row of $M^n$ is $r_i(M^n)\geq 4$. By Theorem~\ref{spectral}, $4\leq \rho(M^n)\leq \rho(M)^n$. In consequence, $\rho(M)\geq \sqrt[n]{4}$ and the result follows.
\end{proof}

\section{Proof of Theorem~\ref{Tachin}}\label{S10}
Now we have all the necessary ingredients to deploy the proof of Theorem~\ref{Tachin} as sketched in Section~\ref{S5}.

\begin{proof}[Proof of Theorem~\ref{Tachin}]
We prove the result by induction on the period $n$. For $n=3$ there is nothing to prove, since the only pattern with positive entropy is $\mathcal{Q}_3$. Let $\mathcal{P}$ an $n$-periodic pattern and assume now that the theorem is true for any period less than $n$. By Theorem~\ref{5.3}, we can assume that all openings of $\mathcal{P}$ are zero entropy patterns.

If $\mathcal{P}$ is $\pi$-reducible for a basic path $\pi$, then, by Proposition~\ref{twocharact}, it has a separated structure of $p\ge2$ trivial blocks. The associated skeleton $\mathcal{S}$ is a $p$-periodic pattern and, by Corollary~\ref{skeletonok}, its entropy is the same as $\mathcal{P}$, positive. In particular, $p\ge3$. Since $p$ is a strict divisor of $n$, $h(\mathcal{S})\ge\log(\lambda_p)$ by the induction hypothesis. Then, $h(\mathcal{P})>\log(\lambda_n)$ by Proposition~\ref{81} and we are done in this case.

From now on, we will assume that $\mathcal{P}$ is $\pi$-irreducible. By Theorem~\ref{2inners_3op}, $\mathcal{P}$ is either a $k$-flower or a triple chain.

Assume first that $\mathcal{P}$ is a $k$-flower. If $k=2$, then we are done by Theorem~\ref{fwd}. Otherwise, by Corollary~\ref{1inner}, $\mathcal{P}$ has a subordinated $n'$-periodic pattern $\mathcal{P}'$ with positive entropy. Since $n'$ is a strict divisor of $n$, $h(\mathcal{P}')\ge\log(\lambda_{n'})$ by the induction hypothesis. Therefore, $h(\mathcal{P})>\log(\lambda_n)$ by Lemma~\ref{subor}.

Finally, we are left with the case that $\mathcal{P}$ is a triple chain. Since we are in the hypotheses of Corollary~\ref{n4}, then $h(\mathcal{P})>\log(\sqrt[n]{4})$. By Proposition~\ref{propietats}(b), we are also done in this case and thus the theorem follows.
\end{proof}

\section{Proof of Corollary~\ref{Tachin2} and Theorem~\ref{Tachin3}}\label{S11}
In this section we formally define the subfamily $\Irr_n\subset\Pos_n$ of all \emph{irreducible} $n$-periodic patterns.
Then we recall some results relating reducibility and entropy, and finally prove Corollary~\ref{Tachin2} and Theorem~\ref{Tachin3}.

Let $\mathcal{P}$ be an $n$-periodic pattern. We say that $\mathcal{P}$ is \emph{reducible} if it has a block
structure. Otherwise, $\mathcal{P}$ will be said to be \emph{irreducible}.
From the characterization of zero entropy
patterns given by Proposition~\ref{56}, it follows that any irreducible pattern has positive entropy, so that $\Irr_n\subset\Pos_n$.
The next result states that $\Pos_n=\Irr_n$ when either $n$ is a prime or $n=4$.

\begin{lemma}\label{Red}
Any $n$-periodic pattern with positive entropy is irreducible if either $n$ is a prime or $n=4$.
\end{lemma}
\begin{proof}
Let $\mathcal{P}$ be an $n$-periodic pattern with $h(\mathcal{P})>0$. If $n$ is a prime, then $\mathcal{P}$ cannot
be reducible since, by definition, the period of any pattern with a block structure has strict divisors. Assume that
$n=4$ and that $\mathcal{P}$ is reducible. In this case, the only possible block structure for $\mathcal{P}$ is a
separated 2-block structure of 2 trivial blocks. The corresponding skeleton is the trivial pattern of 2 points,
with entropy zero. By Proposition~\ref{81}, $h(\mathcal{P})=0$, a contradiction.
\end{proof}

Let us see that the patterns $\mathcal{Q}_n$ with minimum positive entropy (see Figure~\ref{Q6}) are irreducible.

\begin{lemma}\label{qnirred}
Let $n\ge3$ be a positive integer. Then, the pattern $\mathcal{Q}_n$ is irreducible.
\end{lemma}
\begin{proof}
Let $(T,P,f)$ be a model of $\mathcal{Q}_n$ and let $P=\{x_i\}_{i=0}^{n-1}$ be time labeled.
Recall that $\{x_0,x_{n-1}\}$ is an extremal discrete component of $\mathcal{Q}_n$, with $x_{n-1}$ being an endpoint.
If $\mathcal{Q}_n$ has a $p$-block structure $P_0\cup P_1\cup\ldots\cup P_{p-1}$ with $p\ge2$, then the fact that
$\chull{P_i}\cap P_j=\emptyset$ for $i\ne j$ implies that the block containing $x_{n-1}$ (say, $P_k$) should contain also $x_0$.
Since $f(x_{n-1})=x_0$, if follows that $f(P_k)\cap P_k\ne\emptyset$, in contradiction with the definition of a
block structure.
\end{proof}

Now we are ready to prove Corollary~\ref{Tachin2}.

\begin{proof}[Proof of Corollary~\ref{Tachin2}]
It follows trivially from Lemma~\ref{qnirred} and Theorem~\ref{Tachin}.
\end{proof}

Let us proceed now with Theorem~\ref{Tachin3}, that gives the minimum positive entropy when we restrict ourselves to the family
of reducible $n$-periodic patterns. By Lemma~\ref{Red}, the problem makes sense only when $n$ is a composite integer larger than 5.
As we will see and in contrast to what happens for irreducible patterns, the minimum entropy reducible pattern is not unique.

\begin{lemma}\label{minim}
Let $n\ge6$ be a composite integer and let $p$ be the smallest proper divisor of $n$.
\begin{enumerate}
\item The minimum value in the set $\{(1/d)\log(\lambda_{n/d}): d\mbox{ divides $n$ and } d\ne 1,n\}$
is attained when $d=p$.
\item If $n>6$, then $4^{1/n}>(\lambda_{n/p})^{1/p}$.
\end{enumerate}
\end{lemma}
\begin{proof}
Let us prove (a). It suffices to show that $(\lambda_{n/d})^{1/d}$ is minimum when $d=p$. Since $\lambda_i>1$ for any
$i\ge3$, this is equivalent to show that $(\lambda_{n/d})^{n/d}$ is minimum when $d=p$. This claim will be
true if we prove that $(\lambda_k)^k$ is decreasing in $k$. By definition, $\lambda_i$ satisfies $(\lambda_i)^i-2\lambda_i-1=0$.
On the other hand, Proposition~\ref{propietats}(a) tells us that $\lambda_k$ decreases with $k$. Putting all together yields
$(\lambda_k)^k=2\lambda_k+1>2\lambda_{k+1}+1=(\lambda_{k+1})^{k+1}$.

Let us prove (b). We have to show that $4^{1/n}>(\lambda_{n/p})^{1/p}$, which is equivalent to prove that
$4>(\lambda_{n/p})^{n/p}=2\lambda_{n/p}+1$. This will be true if $3/2>\lambda_{n/p}$. Now observe that
$n\ge8$ and $n/p\ge4$ because $n>6$ is composite. Since $\lambda_k$ decreases with $k$,
$\lambda_{n/p}\le\lambda_4\approx1.39$ and we are done.
\end{proof}

We will prove Theorem~\ref{Tachin3} in two steps. First, we will show that if $\mathcal{P}$ is an $n$-periodic and
reducible pattern with positive entropy, then $h(\mathcal{P})\ge\log(\lambda_{n/p})/p$, where $p$ is the smallest prime factor of $n$.
Secondly, we will provide examples of patterns attaining precisely this entropy.

To find examples of reducible patterns with minimum positive entropy, we use the following construction, a
generalization of the classic notion of \emph{extension} for interval patterns \cite{block2}. Let $\mathcal{R}$
be a $k$-periodic pattern and let $p\ge2$ be an integer. A \emph{$p$-extension} of $\mathcal{R}$ is a
$pk$-periodic pattern $\mathcal{P}$ such that, for any model $(T,P,f)$ of $\mathcal{P}$, there is a separated
$p$-block structure $P=P_0\cup P_1\cup\ldots\cup P_{p-1}$ satisfying:
\begin{enumerate}
\item The associated skeleton is a trivial $p$-periodic pattern.
\item The pattern $([\chull{P_0},P_0],[f^p])$ is $\mathcal{R}$.
\item For $1\le i<p$, the pattern $([\chull{P_i},P_i],[f^p])$ is trivial.
\end{enumerate}
There are several possible $p$-extensions of a given pattern, see Figure~\ref{exten} for an example. The next result states that,
in any case, all the possible $p$-extensions have the same entropy. This fact is well known for extensions of interval patterns and other
similar constructions, and the same proof applies in this setting. See for instance Lemma~4.4.16 of \cite{biblia}.

\begin{lemma}\label{hext}
If $\mathcal{P}$ is a $p$-extension of a pattern $\mathcal{R}$, then $h(\mathcal{P})=h(\mathcal{R})/p$.
\end{lemma}

\begin{figure}
\centering
\includegraphics[scale=0.65]{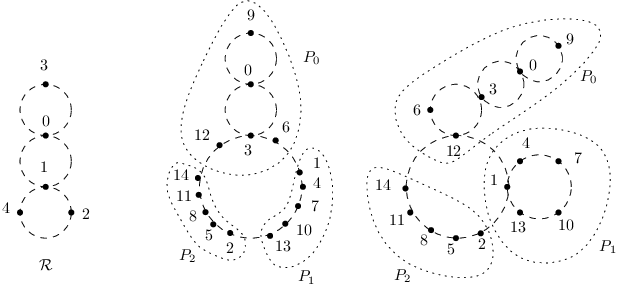}
\caption[fck]{Two different 3-extensions of a 5-periodic pattern $\mathcal{R}$. In both cases, the patterns of $f^3$ in each block are either
trivial or $\mathcal{R}$ itself, and the skeleton is the trivial 3-periodic pattern.\label{exten}}
\end{figure}

Now we are ready to prove Theorem~\ref{Tachin3}.

\begin{proof}[Proof of Theorem~\ref{Tachin3}]
Let $\mathcal{P}$ be an $n$-periodic and reducible pattern with positive entropy. We claim first that
$h(\mathcal{P})\ge\log(\lambda_{n/p})/p$, where $p$ is the smallest prime factor of $n$.

Recall (Lemma~\ref{openingblock}) that after performing an opening on $\mathcal{P}$, the obtained pattern is also reducible,
and, by Theorem~\ref{5.3}, its entropy is less or equal to $h(\mathcal{P})$. Therefore, from now on we can assume that $\mathcal{P}$
satisfies the property (\ref{opening0}) introduced in page~\pageref{pagina}.

Assume that $\mathcal{P}$ is $\pi$-reducible for some basic path $\pi$. By Proposition~\ref{twocharact},
$\mathcal{P}$ has a separated block structure of $k\ge2$ trivial blocks, where $k$ is a strict divisor of $n$. Moreover,
if $\mathcal{S}$ is the corresponding skeleton, then $h(\mathcal{S})=h(\mathcal{P})$ by Proposition~\ref{81}.
In particular, since $\mathcal{S}$ is $k$-periodic and $h(\mathcal{P})>0$, it follows that $k>2$.
Now, Theorem~\ref{Tachin} tells us that $h(\mathcal{S})\ge\log(\lambda_k)$. If we set $d:=n/k$, we have that
\[ h(\mathcal{P})=h(\mathcal{S})\ge\log(\lambda_k)=\log(\lambda_{n/d})>\log(\lambda_{n/d})/d, \]
which, from (a) of Lemma~\ref{minim}, is larger than or equal to $\log(\lambda_{n/p})/p$. So, the claim is proved
when $\mathcal{P}$ is $\pi$-reducible.

From now on we assume that $\mathcal{P}$ is $\pi$-irreducible. Theorem~\ref{2inners_3op} and (\ref{opening0}) imply
then that $\mathcal{P}$ is either a $k$-flower of a triple chain.

Assume first that $\mathcal{P}$ is a $k$-flower. Let $(T,P,f)$ be a model of $\mathcal{P}$ and let $x$ be the only inner point. Since $\mathcal{P}$ is
reducible, it has a $d$-block structure. Let $P_0$ be the block containing $x$. Note that, by the definition of a block structure
and the fact that $x$ is the unique inner point, all the remaining blocks are trivial. Consider the $(n/d)$-periodic pattern
$\mathcal{R}:=([\chull{P_0}_T,P_0],[f^d])$. Since $h(\mathcal{R})$ is smaller than or equal to the entropy of any map exhibiting
$\mathcal{R}$ and $f^d$ exhibits $\mathcal{R}$,
\begin{equation}\label{desi}
h(\mathcal{R})\le h(f^d)=d\cdot h(f)=d\cdot h(\mathcal{P}).
\end{equation}
Observe that if there exists some basic path $\pi\subset P_0$ of $\mathcal{R}$ that never splits by $f^d$, then $\pi$ never splits by $f$,
because all blocks different from $P_0$ are trivial. In this case $\mathcal{P}$ would be $\pi$-reducible, a contradiction. In
consequence, $\mathcal{R}$ is $\pi$-irreducible. In particular, $h(\mathcal{R})>0$. Thus, by Theorem~\ref{Tachin}, $h(\mathcal{R})\ge\log(\lambda_{n/d})$.
Using (\ref{desi}) yields $h(\mathcal{P})\ge\log(\lambda_{n/d})/d$ which, from (a) of Lemma~\ref{minim}, is larger than or equal to
$\log(\lambda_{n/p})/p$. So, the claim is proved in this case.

Finally, assume that $\mathcal{P}$ is a triple chain. We treat first the special case $n=6$. In this case, $\mathcal{P}$
cannot have a 3-structure of blocks of two points, since $\mathcal{P}$ would be $\pi$-reducible. So, the only
possibility is that $\mathcal{P}$ has a 2-structure of blocks of three points. Again, if both blocks were
trivial, $\mathcal{P}$ would be $\pi$-reducible. So, in at least one of the two blocks, $P_0$, the pattern
$([\chull{P_0}_T,P_0],[f^2])$ is a non-trivial 3-periodic pattern. Such a pattern is uniquely determined and
coincides with the 3-periodic \v{S}tefan cycle of the interval \cite{ste}, with entropy $\log(\lambda_3)$.
Using the same argument as in the previous paragraph, the claim follows also in the case $n=6$.

Finally, if $n>6$, since $\mathcal{P}$ is $\pi$-irreducible and (\ref{opening0}) holds,
then Corollary~\ref{n4} tells us that $h(\mathcal{P})>\log(\sqrt[n]{4})$, larger than or equal to
$\log(\lambda_{n/p})^{1/p}$ by Lemma~\ref{minim}(b). The claim, thus, holds also in this case.

Once the claim that gives a lower bound for the entropy of $\mathcal{P}$ is proven, we have to give an example of a
reducible $n$-periodic pattern having precisely this entropy. By Lemma~\ref{hext}, it is enough to consider any $p$-extension of
$\mathcal{Q}_{n/p}$.
\end{proof}


\begin{thebibliography}{10}
\bibitem{AKM}
R.~L. Adler, A.~G. Konheim, M.~H. McAndrew, \emph{Topological entropy},
  Trans. Amer. Math. Soc. \textbf{114} (1965), 309--319. \MR{0175106 (30
  \#5291)}

\bibitem{patgraf}
Ll. Alsed{\`a}, F.~Gautero, J.~Guaschi, J.~Los, F.~Ma{\~n}osas, P.~Mumbr{\'u},
\emph{Patterns and minimal dynamics for graph maps},
Proc. London Math. Soc. (3) \textbf{91} (2005), no.~2, 414--442.

\bibitem{ajk}
Ll. Alsed{\`a}, D.~Juher, D.M. King, \emph{A lower bound for the maximum topological
entropy of {$(4k+2)$}-cycles}, Experiment. Math. \textbf{17} (2008), no.~4, 391--407.

\bibitem{ajkm}
Ll. Alsedà, D. Juher, D. M. King, F. Mañosas, \emph{Maximizing entropy of cycles on trees},
Discrete Contin. Dyn. Syst. {\bf 33}(8) (2013), 3237--3276.

\bibitem{reducibility}
Ll. Alsed\`{a}, D. Juher, F. Ma\~{n}osas, \emph{Topological and algebraic
reducibility for patterns on trees}, Ergodic Theory Dynam. Systems, 35 (2015), 34--63.

\bibitem{aglmm}
Ll. Alsed\`{a}, J. Guaschi, J. Los, F. Ma\~{n}osas, P. Mumbr\'{u}, \emph{Canonical representatives for
  patterns of tree maps}, Topology \textbf{36} (1997), no.~5, 1123--1153.
  \MR{1445556 (99f:58062)}

\bibitem{ajm}
Ll. Alsed\`{a}, D. Juher, F. Ma\~nosas, \emph{On the minimum positive entropy for cycles on trees},
Transactions Amer. Math. Soc. {\bf 369}(1) (2017), 187--221.

\bibitem{forward}
Ll. Alsed\`a, D. Juher, F.  Ma\~{n}osas, \emph{Forward triplets and topological entropy on trees},
Discrete Contin. Dyn. Syst. 42 (2022), no. 2, 623-–641.

\bibitem{biblia}
Ll. Alsed\`{a}, J. Llibre, M. Misiurewicz,
  \emph{Combinatorial dynamics and entropy in dimension one}, second ed.,
  Advanced Series in Nonlinear Dynamics, vol.~5, World Scientific Publishing
  Co. Inc., River Edge, NJ, 2000. \MR{1807264 (2001j:37073)}

\bibitem{AMM}
Ll. Alsed{\`a}, F.~Ma{\~n}osas, P.~Mumbr{\'u}, \emph{Minimizing
topological entropy for continuous maps on graphs}, Ergodic Theory Dynam.
Systems \textbf{20} (2000), no.~6, 1559--1576.

\bibitem{ay1}
Ll. Alsed\`{a}, X. Ye, \emph{No division and the
set of   periods for tree maps}, Ergodic Theory Dynam. Systems
\textbf{15} (1995), no.~2, 221--237. \MR{1332401 (96d:58109)}

\bibitem{bald} (913178 (89c:58057))
S.~Baldwin, \emph{Generalizations of a theorem of {S}arkovskii on
orbits of continuous real-valued functions}, Discrete Math. \textbf{67}
(1987), no.~2, 111--127.

\bibitem{bald2}
S. Baldwin, \emph{Toward a theory of forcing on maps of trees}, Proceedings of the
Conference ``Thirty Years after Sharkovski\u\i's Theorem: New Perspectives''
(Murcia, 1994), Internat. J. Bifur. Chaos Appl. Sci. Engrg. {\bf 5} (1995), no. 5,
1307–-1318.

\bibitem{bern}
C. Bernhardt, \emph{A Sharkovsky theorem for vertex maps on trees}, J. Difference Equ. Appl. {\bf 17} (2011), no. 1, 103-–113.

\bibitem{blan}
F.~Blanchard, E.~Glasner, S.~Kolyada, A~Maas, \emph{On Li-Yorke pairs}, J. Reine Angew. Math. {\bf 547} (2002), 51--68.

\bibitem{bgmy}
L. Block, J. Guckenheimer, M. Misiurewicz, L.S. Young, \emph{Periodic
points and topological entropy of one-dimensional maps}, Global theory
of dynamical systems, pp. 18--34, SLNM {\bf 819}, Springer, Berlin.

\bibitem{block2}
L. Block, \emph{Simple periodic orbits of mappings of the
interval}, Trans. Amer. Math. Soc. \textbf{254} (1979), 391--398.
\MR{539925 (80m:58031)}

\bibitem{Blokh}
A. Blokh, \emph{Trees with snowflakes and zero entropy maps},
Topology \textbf{33} (1994), 379--396.
\MR{1273790 (95b:58119)}


\bibitem{bow}
R. Bowen, \emph{Entropy and the fundamental group}, 21--29, Lecture Notes in Mathematics
{\bf 668}, Springer, Berlin, 1978.

\bibitem{GT}
W. Geller, J. Tolosa, \emph{Maximal entropy odd orbit types}, Trans. Amer. Math. Soc. \textbf{329} (1992), 161--171.
\MR{1020040 (92e:58163)}

\bibitem{GZ}
W. Geller, Z. Zhang, \emph{Maximal entropy permutations of even size}, Proc. Amer. Math. Soc. \textbf{126} (1998), 3709--3713.
\MR{1458873 (99b:58078)}

\bibitem{fat}
A. Fathi, F. Laudenbach, V. Poenaru, \emph{Travaux de Thurston sur les surfaces}, Asterisque {\bf 66--67} (1979).

\bibitem{fm}
J. Franks, M. Misiurewicz, \emph{Cycles for disk homeomorphisms and thick trees},
Nielsen theory and dynamical systems (South Hadley, MA, 1992), 69–-139, Contemp. Math. {\bf 152} (1993).

\bibitem{KS}
D.M. King, J. B. Strantzen, \emph{Maximum entropy of cycles of even period}, Mem. Amer. Math. Soc. \textbf{152} (2001), no.~723, viii+59.

\bibitem{lmpy}
T.Y. Li, M. Misiurewicz, G. Pianigiani, J.A. Yorke,
  \emph{No division implies chaos}, Trans. Amer. Math. Soc. \textbf{273}
  (1982), no.~1, 191--199. \MR{664037 (83i:28024)}

\bibitem{li-yorke}
T.-Y. Li, J.~Yorke, \emph{Period three implies chaos}, American Mathematical Monthly {\bf 82} (1975),
985--992.

\bibitem{mat}
T. Matsuoka, \emph{The number and linking of periodic solutions of periodic systems},
Invent. Math. {\bf 20} (1983), 319--340

\bibitem{Minc}
H. Minc, \emph{Nonnegative Matrices}, John Wiley \& Sons Inc., New York, 1988.

\bibitem{minor}
M. Misiurewicz, \emph{Minor cycles for interval maps}, Fund. Math. {\bf 145} (1994), no. 3, 281--304.

\bibitem{mn} (1086562 (92h:58105))
M.~Misiurewicz, Z.~Nitecki, \emph{Combinatorial patterns for maps of the interval},
Mem. Amer. Math. Soc. \textbf{94} (1991), no.~456, vi+112.

\bibitem{shar}
O.~M. Sharkovskii, \emph{Co-existence of cycles of a continuous mapping of the line into itself}, Ukrain. Mat.
\u Z. \textbf{16} (1964), 61--71. \MR{0159905 (28 \#3121)}

\bibitem{ste}
P. \v{Stefan}, \emph{A theorem of \v{S}arkovskii on the existence of periodic orbits of continuous
endomorphisms of the real line}, Comm. Math. Phys. {\bf 54} (1977), 237--248.

\bibitem{Thu}
W.~P. Thurston, \emph{On the geometry and dynamics of diffeomorphisms of surfaces},
Bull. Amer. Math. Soc. (N.S.) \textbf{19} (1988), 417--431. \MR{956596 (89k:57023)}

\end{thebibliography}
\end{document}